\documentclass[12pt]{article}
\usepackage[utf8]{inputenc}
\usepackage{times}
\usepackage[centertags]{amsmath}
\usepackage{amssymb}
\usepackage{amsthm}
\usepackage{enumitem}
\usepackage[utf8]{inputenc}
\usepackage[T1]{fontenc}
\usepackage[colorlinks=true,urlcolor=red,citecolor=red,linkcolor=red]{hyperref}
\usepackage[a4paper, top=5cm, bottom=5cm, left=3.5cm, right=2.5cm]{geometry}
\usepackage[italian,english]{babel}
\usepackage{float}
\usepackage{mathtools}
\usepackage{graphicx}
\usepackage{caption}
\usepackage{subcaption}
\usepackage{tikz}
\usepackage{tikz-3dplot}
\usepackage{qtree}
\usepackage{tikz-qtree}
\usepackage[lined,algonl,boxed]{algorithm2e}
\usepackage{pgfplots}
\pgfplotsset{compat=1.13}
\usetikzlibrary{trees,positioning,shapes}
\usetikzlibrary{patterns}
\usepackage{verbatim}
\newtheorem{theorem}{Theorem}[section]
\newtheorem{definition}[theorem]{Definition}
\newtheorem{prop}[theorem]{Proposition}
\newtheorem{lemma}[theorem]{Lemma}
\newtheorem{setting}[theorem]{Setting}

\theoremstyle{definition}
\newtheorem{oss}[theorem]{Remark}

\theoremstyle{definition}
\newtheorem{ex}[theorem]{Example}

\newcommand{\bs}{\boldsymbol}

\newcommand{\al}{\boldsymbol{\alpha}}
\newcommand{\be}{\boldsymbol{\beta}}
\newcommand{\de}{\boldsymbol{\delta}}
\newcommand{\e}{\boldsymbol{e}}
\newcommand{\om}{\boldsymbol{\omega}}
\newcommand{\g}{\boldsymbol{\gamma}}
\newcommand{\te}{\boldsymbol{\theta}}
\newcommand{\he}{\boldsymbol{\eta}}
\newcommand{\ga}{\boldsymbol{\gamma}}
\newcommand{\eps}{\boldsymbol{\epsilon}}
\newcommand{\C}{\boldsymbol{c}}
\newcommand{\Ap}{\mathrm{\bf Ap}}
\newcommand{\wt}{\hspace{0.1cm}\widetilde{\wedge}\hspace{0.1cm}}
\newcommand{\wh}[1]{\widehat{#1}}
\newcommand{\bei}[1]{\bs{\beta^{(#1)}}}
\newcommand{\td}[2]{\widetilde\Delta_{#1}^S(#2)}
\newcommand{\ds}[2]{\Delta_{#1}^S(#2)}

\newcommand{\tdE}[2]{\widetilde\Delta_{#1}^E(#2)}

\newcommand{\rc}{\color{red}}  %% COLORE ROSSO
\newcommand{\ec}{\color{black}} %% COLORE NERO
\newcommand{\bc}{\color{blue}}
\newcommand{\mc}{\color{magenta}}

\newcommand{\N}{\mathbb{N}}

\def\u{\underline}
\def\ee{{\underline{e}}}
\def\OO{{\mathcal O}}
\def\CC{{\mathcal C}}
\def\w{\widetilde}
\def\k{\overline}

{\null\vfill\begin{center}%
		\bfseries Acknowledgements\end{center}}%
{\vfill\null}
\newcommand{\keywords}[1]{\emph{Keywords:} #1}
\newcommand{\MSC}[1]{\emph{Mathematics Subject Classification 2020:} #1}

\newcounter{lastnote}

\begin{document}

\title{The value semigroup of a plane curve singularity with several branches} 

\author{M. D'Anna\thanks{Università degli studi di Catania, Dipartimento di Matematica e Informatica, Catania \emph{e-mail: marco.danna@unict.it}},  F. Delgado\thanks{Universidad de Valladolid. IMUVa (Institut of research in Mathematics). \emph{fdelgado@uva.es}}, L. Guerrieri\thanks{Jagiellonian University, Instytut Matematyki, 30-348 Krak\'{o}w \emph{e-mail: lorenzo.guerrieri@uj.edu.pl}}, N. Maugeri\thanks{Università degli studi di Catania, Dipartimento di Matematica e Informatica, Catania \emph{e-mail: nicola.maugeri@unict.it}}, V. Micale\thanks{Università degli studi di Catania, Dipartimento di Matematica e Informatica, Catania \emph{e-mail: vmicale@dmi.unict.it}
%\newline 
}
}

\date{}

\linespread{1,0}
\maketitle
\begin{abstract}
\noindent  We present a constructive procedure, based on the notion of Ap\'ery set, to 
obtain the value semigroup of a plane curve singularity from the value semigroup of its blow-up and vice-versa. 
In particular we give a blow-down process that allows to reconstruct a plane algebroid curve form its blow-up, even if it is not local. Then we characterize numerically all the possible multiplicity trees of plane curve singularities,
obtaining in this way a constructive description of all their value semigroups.
\end{abstract}
\\
\MSC{14H20, 14H50, 13H15,  20M14}

%\begin{document}

\maketitle
\section{Introduction} 

An algebroid branch is a ring of the form $\mathcal O= K[[X_1, \dots, X_n]]/P$, where $K$ is an algebraically closed field and $P$ is an height $n-1$ prime ideal. Algebroid branches naturally appear in the study of curve singularities,
as completions of the local rings associated to a singular point of an algebraic curve, with one branch in that point. Since Zariski \cite{zariski}, a classical tool to study and classify singularities is given by the value semigroup
associated to an algebroid branch: in fact, the integral closure $\overline {\mathcal O}$ of ${\mathcal O}$ in its quotient field is a DVR isomorphic to $K[[t]]$. Hence every nonzero element $g\in \mathcal O$ has
a value $v(g):= ord_t(g) \in \N$ and the set of values of its elements constitute a
numerical semigroup $v(\mathcal O)=S$, i.e., a submonoid of $\N$ with
finite complement in it. 
The knowledge of the value semigroup gives many information on the ring $\mathcal O$;
for example its smallest nonzero value is the multiplicity $e(\mathcal O)$ of the singularity, and from the value semigroup one can easily compute the degree of singularity (i.e. the length $l_{\mathcal O} (\overline {\mathcal O}/\mathcal O)$), or one can check the Gorenstein and the complete intersection properties.

Another classical invariant to classify a branch singularity is given by the
sequence of multiplicities of the successive blow-ups of $\mathcal O$,
$(e(\mathcal O),e(\mathcal O'),e(\mathcal O^{(2)}),\ldots)$ (see e.g. \cite{zariski}).
Two algebroid branches are said formally equivalent it they share the same sequence of multiplicities; 
in general, the value semigroups and the multiplicity sequence are independent criteria of equisingularity.

If we want to consider a curve singularity with $d$ branches, we have to deal with algebroid curves, i.e. rings of the form
$\mathcal O=K[[X_1, \dots,X_n]]/P_1\cap \dots \cap P_d$, where the $P_i$ are pairwise distinct prime ideals of heigth $n-1$ and determine the branches. In this case the integral closure of $\mathcal O$ in its total ring of fractions is a product of DVR's,
$\overline {\mathcal O} \cong K[[t_1]]\times\cdots\times K[[t_d]]$, 
where $K[[t_i]]$ is the integral closure $\overline{\CC}_i$  of the i-th branch $\CC_i := K[[X_1,...,X_n]]/P_i$ 
and the set of
values $S=v(\mathcal O)$ is a submonoid of $\mathbb N^d$ (here $v(g)=(v_1(g), \ldots, v_d(g)) \in \mathbb N^d$, where $v_i$ is the valuation of the i-th branch). The projections $S_i$
of $S$ on the coordinate axes are the value semigroups of each
branch.

Again, as for the one branch case,
the value semigroup gives many informations on the singularity.
But, while any numerical semigroup is the value semigroup of a one branch singularity, there is no characterization of the semigroups appearing as value semigroups 
of algebroid curves with $d > 1$ branches.\\

%Two algebroid curves $\mathcal O=k[[X_1, \dots,X_n]]/P_1\cap \dots \cap P_d$
%and $\mathcal C=k[[X_1, \dots,X_n]]/Q_1\cap \dots \cap Q_d$ are said to be formally
%equivalent if (after a renumbering of the branches)
%$k[[X_1, \dots,X_n]]/P_i$ and $k[[X_1, \dots,X_n]]/Q_i$ have the same multiplicity sequence for $i=1,\ldots,d$ and if
%the intersection multiplicities are the same for all pairs 
%$(i,j), \ i\ne j$ {\bc CHECK}

In this article we consider the case of plane curve singularities. 
When we have only one branch,
there are classical
characterizations for the possible 
numerical semigroups that are value semigroups of an algebroid plane branch (that now is a ring of the form
$\mathcal O=K[[X,Y]]/(F)$, with $F$ irreducible). Moreover, it is well-known that the value semigroup and the multiplicity sequence of an algebroid branch  become two equivalent criteria of equisingularity (see \cite{zariski})
and in fact it is possible to reconstruct the multiplicity sequence from the value semigroup and vice-versa. 
More precisely, in \cite {aperyold}, Ap\'ery considered a particular generating set of $v(\mathcal O)$,
called the Ap\'ery set, and showed that one can compute the Ap\'ery set of the value semigroup $v(\mathcal B(\mathcal O))$ of the blowup of $\mathcal O$ from that of $v(\mathcal O)$,
and vice-versa.
This is the reason why, for plane branches, the value semigroup
 and the multiplicity sequence are two equivalent sets of invariants.
In \cite{planealgebroid} it has been shown how to use Apery's result to easily obtain the
value semigroup from the multiplicity sequence and vice-versa.
It is worth noticing that, if we instead
consider plane analytic branches, these invariants (value semigroup or multiplicity sequence) determine the
topological class of the branch (see again \cite{zariski}).

If we want to study a plane curve singularity with $d >1$ branches, we have to deal with plane algebroid curves, i.e. rings of the form
$\mathcal O=K[[X,Y]]/(H_1\cdots H_d)$, where the $H_i$ are irreducible and pairwise coprime. 
In this case two plane algebroid curves $\mathcal O=K[[X,Y]]/(H_1\cdots H_d)$
and $\mathcal Q=K[[X,Y]]/(G_1\cdots G_d)$ are formally
equivalent if  (after a renumbering of the branches)
the branches $\CC_i=K[[X,Y]]/(H_i)$ and $\mathcal D_i=K[[X,Y]]/(G_i)$ have the same multiplicity sequence for $i=1,\ldots,d$ and if
the intersection multiplicities 
$[\CC_i,\CC_j]:=l_{\mathcal O}(K[[X,Y]]/(H_i,H_j))$ and
$[\mathcal D_i,\mathcal D_j]:=l_{\mathcal Q}(K[[X,Y]]/(G_i,G_j))$  (where $l$ denotes the lenght of a module over a ring) are the same for all pairs
$(i,j), \ i\ne j$. Waldi has shown in \cite{waldi} that two plane algebroid
curves are formally equivalent if and only if they have the same value
semigroup.

Hence it is natural to ask whether it is possible to characterize the value semigroup of a plane singularity with more than one branch and to investigate how to reconstruct it by the multiplicity sequences and the intersection multiplicities of its branches and vice-versa.

The problem of the computation of the semigroup of values for $d>1$ (and as a
consequence of its characterization in some terms) from the semigroups of each
branch together with the intersection multiplicities between pair of branches was
resolved in
\cite{felix} following the next inductive way.
Assume that one knows the semigroups of less than $d$ branches, i.e. the semigroups
$S_J$ of the proper subset of branches corresponding to $J\subset
\{1,\ldots,d\}$, $\# J<d$. Then one can compute $S$ from the subsemigroups
$\{S_J \,| \, \#J = d-1\}$ and a finite set of elements
$B=\{\beta^1, \ldots, \beta^m\}\subset S$ (the generalization of maximal contact values, i.e.
of the minimal set of generators of the case $d=1$). The set $B$ can be computed
explicitly from the semigroups $S_i$, $i=1,\ldots, d$ and the intersection
multiplicities of pairs of branches. It must be noticed that this way was made for the
case $d=2$ by Garc\'{\i}a in \cite{garcia} and Bayer in \cite{bayer}.\\

However the above description is not easy, among other things demands inductively the
computation of the projections $S_J$; moreover it is not established in terms of the
resolution process, which is a very natural way to understand the plane curve
singularities.

This different approach to the problem was addressed and solved in \cite{apery} for the two
branches case and for characteristic $0$.
In that paper the authors use two main tools: 
firstly they show how to encode the data that determine formal equivalence in a tree, that they call multiplicity tree; secondly they define the Ap\'ery set of the value semigroup 
(which is now an infinite set)
and make a partition of it in "levels", describing them as value sets of particular elements of the algebroid curve. Then they show that, in case $\mathcal O$ and its blow-up $\mathcal B(\mathcal O)$ are both local,
the levels of the Ap\'ery sets of their value semigroups can be obtained one from the other.
Using these tools and a result of Garcia \cite{garcia} (that holds only in the two branch case), they show how to obtain the value semigroup from the multiplicity tree and vice-versa; 
this fact, together with a numerical 
description of the admissible multiplicity trees, gives a constructive characterization of the value semigroups of a plane singularity with two branches.\\

The aim of this paper is to generalize this approach to any number of branches, without restrictions on the characteristic. 
There are two main problems that arise. The first one is the fact that the definition of the partition of the Ap\'ery set given in \cite{apery} does not work in more than two branches and in the non local case. This problem has been addressed and solved in \cite{DGM}, \cite{GMM} and \cite{GMM2}, where a new 
definition of the levels of the Ap\'ery set, that works well in general, has been given;
moreover, in \cite{GMM2} the authors show that this new definition agree with the old one in the two branch local case.

The second problem derives from the fact that blowing up the algebroid curve, at some point (i.e. when at least two branches have different tangents) the blow-up is no more local. Our aim is to 
obtain a procedure to obtain the Ap\'ery set of $v(\mathcal O)$ from the Ap\'ery set the value semigroup $v(\mathcal B(\mathcal O))$ of its blow-up 
and vice-versa; to do this we can make use of the new definition of levels of the Ap\'ery set that holds also in the non local case.
Moreover, we also need to show, for any number of branches, that the levels of the Ap\'ery set 
can be obtained as value sets of particular sets
of elements of $\mathcal B(\mathcal O)$, also in the non local case. 
And since $\mathcal B(\mathcal O)$ is not local, we cannot anymore present it as a quotient
of $K[[X,Y]]$, as it was done in \cite{apery}.

Hence our main task is to prove 
Theorems \ref{semilocalring1part} and \ref{semilocalring2part}, where we show
in the general case (i.e. for any number of branches, in the semilocal case and with no restrictions on the characteristic)
how to describe the levels of the Ap\'ery set.
After doing that, we can give the searched procedure (see Theorem \ref{final}). In order to obtain it, we prove at ring level a procedure that, starting 
by a product $\mathcal V$ of local rings of plane algebroid curves, produces 
a local ring $\mathcal U$ of a plane algebroid curve, such that $\mathcal B(\mathcal U)=\mathcal V$ (Proposition \ref{blowup}). So, we have a sort of blow-down process that reverses the blow-up: in fact, if we start by a plane algebroid curve $\mathcal O$, we blow it up and then blow-down $\mathcal B(\mathcal O)$, we get again $\mathcal O$ (Proposition \ref{U=O}).

Now, in order to obtain a constructive characterization of the value semigroup of a plane curve singularity, it remains to characterize numerically the admissible multiplicity trees of a curve singularity with any number of branches; 
this is classically known for the one-branch case, 
it was done for the two branches case 
and characteristic $0$ in
\cite{apery}, and here it is generalized for any number of branches
 without restriction on the characteristic (see Proposition \ref{multtree}).
Using this last result, we can summarize in Theorem \ref{constructS} the equivalence of the following sets of data:
\begin{enumerate}
\item
the semigroup of values $S$ of $\mathcal O$; 
%\item The semigroups $S_i$, $1\le i\le d$, of each branch and the set of intersection
%multiplicities
%$\{[C_i,C_j] | 1\le i < j\le d\}$ between pairs of branches;
\item
The multiplicity tree ${\mathcal T}(R)$ of $R$;
\item the set $E = \{\underline{e}^i = (e^i_0, e^i_1, \ldots); i=1,\ldots, d\}$ of the multiplicity
sequences of the branches $\{\CC_i | 1\le i \le d\}$ plus the splitting numbers
$\{k_{i,j}\}$ between pairs of branches $C_i$, $C_j$; $1\le i < j\le d$.
\end{enumerate}

We now briefly describe the structure of the paper.
Section 2 is devoted to the basic definition about good semigroups; in particular in Definition 2.1 we recall the partition of the Ap\'ery set in levels, fixing the notation in a more convenient way with respect to previous papers. Then we
show that this partition works well 
when both $\mathcal O$ and $\mathcal B(\mathcal O)$ are both local, 
generalizing the arguments of \cite{apery} (see Propositions \ref{discussion}, \ref{3.8apery} and Theorem \ref{4.1apery}).

Section 3 is very technical and contains some new results on the Ap\'ery set, when the semigroup is not local. These results will allow us to find particular elements in the Ap\'ery set, keeping the control on the levels (see e.g Remark \ref{chainofnodes} and Lemmas \ref{subspaceslemma} and \ref{congruencelemma}). 

In Section 4 we extend \cite[Theorem 4.1]{apery} to the case where the blow-up of the coordinate ring of a plane curve is not local. In the first part of the section we describe the level of the Ap\'ery set of the value semigroup of a semilocal ring $R$ as sets of values of specific subsets of $R$ (see Theorems \ref{semilocalring1part} and \ref{semilocalring2part}). In the second part, we describe the blow-down process (Proposition \ref{blowup}) and how the levels of the Ap\'ery set of the value semigroup behave when passing from the ring of a plane curve to its blow-up and vice-versa (Theorem \ref{final}).

Finally, in Section 5 we give a characterization of the admissible multiplicity trees of a plane singularity for any number of branches and independently of the characteristic. To this aim we have to recall the Hamburger-Noether expansion in the one branch case and,
using it, we can generalize the results for the two branches case proved in \cite{apery} for characteristic zero. With an inductive argument we can give the requested characterization for any number of branches (Proposition \ref{multtree}), that leads to Theorem \ref{constructS} and to a constructive characterization of the admissible value semigroups of a plane curve singularity.

\section{Preliminaries on algebroid curves}

To work with value semigroups of algebroid curves we will use the more general concept of good semigroup, introduced in \cite{a-u}. 
Let $\le$ denote the standard componentwise partial ordering in $\mathbb N^d$.
%: 
Given two elements $\al=(\alpha_1, \alpha_2,\ldots,\alpha_d), \be=(\beta_1, \beta_2,\ldots,\beta_d) \in \N^d$,
%then $\boldsymbol{\alpha}\le \boldsymbol{\beta}$ if $\alpha_i\le \beta_i$ for all $i\in \{1,\ldots,d\}$.\\
%Given $\boldsymbol{\alpha},\boldsymbol{\beta}\in \mathbb N^d$, the infimum of the set $\{\boldsymbol{\alpha},\boldsymbol{\beta}\}$ (with respect to $\le$) will be denoted by $\bs{\alpha}\wedge \bs{\beta}$
%Trough this paper, if not differently specified, when referring to minimal or maximal elements of a subset of $\mathbb N^d$, we refer to minimal or maximal elements with respect to $\le$. 
the element $\de$ such that $\delta_i= \min(\alpha_i, \beta_i)$ for every $i=1, \ldots, d$ is called the the infimum of the set $\{\boldsymbol{\alpha},\boldsymbol{\beta}\}$ and will be denoted by $\bs{\alpha}\wedge \bs{\beta}$.

Let $S$ be a submonoid of $(\mathbb N^d,+)$. We say that $S$ is a \emph{good semigroup} if

\begin{itemize}
	\item[(G1)] For every $\boldsymbol{\alpha},\boldsymbol{\beta}\in S$, $\boldsymbol{\alpha}\wedge \boldsymbol{\beta}\in S$;
	\item[(G2)] Given two elements $\boldsymbol{\alpha},\boldsymbol{\beta}\in S$ such that $\al\neq \be$ and $\alpha_i=\beta_i$ for some $i\in\{1,\ldots,d\}$, then there exists $\bs{\epsilon}\in S$ such that $\epsilon_i>\alpha_i=\beta_i$ and $\epsilon_j\geq \min\{\alpha_j,\beta_j\}$ for each $j\neq i$ (and if $\alpha_j\neq \beta_j$ the equality holds).
	\item[(G3)] There exists an element $\boldsymbol{c}\in S$ such that $\boldsymbol{c}+\mathbb N^d\subseteq S$.
\end{itemize}

A good semigroup is said to be \emph{local} if $\boldsymbol{0}=(0,\ldots,0)$ is its
only element with a zero component. 

By (G1) it is always possible to define the element $\boldsymbol{c}:=\min\{\boldsymbol{\alpha}\in \mathbb Z^d\mid \boldsymbol{\alpha}+\mathbb N^d\subseteq S\}$; this element is called \emph{conductor} of $S$.
We set $\boldsymbol{\gamma}:=\boldsymbol{c}-\textbf{1}$.

A subset $E \subseteq \N^d$ is a \it relative ideal \rm of $S$ if $E+S \subseteq E$ and there exists $\al \in S$ such that $\al + E \subseteq S$. A relative ideal $E$ contained in $S$ is simply called an ideal. An ideal $E$
satisfying properties (G1), (G2) is called a \it good ideal \rm (notice that all ideals satisfy (G3) by definition). The minimal element $\boldsymbol{c}_E$ such that $\boldsymbol{c}_E + \N^d \subseteq E$ is called the \it conductor \rm of $E$. As for $S$, we set $\boldsymbol{\gamma}_E:=\boldsymbol{c}_E-\textbf{1}$.

We denote by $\bs{e}=(e_1,e_2,\ldots,e_d)$, the minimal element of $S$ such that $e_i>0$ for all  $i\in \{1,\ldots,d\}$.
The set $\e + S$ is a good ideal of $S$ and its conductor is $\boldsymbol{c} + \e$. Similarly for every $\om \in S$, the principal good ideal $E= \om + S$ has conductor $\boldsymbol{c}_E = \boldsymbol{c} + \om$.

\medskip

Let $\mathcal{O}$ be an algebroid curve with $d$ branches. The value semigroup $S=v(\mathcal{O})$ is a local good semigroup contained in $\N^d$ \cite{a-u}.  In this case, the sum of the coordinates of the element $\e$ is the multiplicity of the curve. Non-local good semigroups may appear as value semigroups of semilocal rings obtained from algebroid curves after blow-ups. General results on good semigroups and value semigroups of curve singularities appear in many papers,
e.g. \cite{a-u}, \cite{c-d-gz}, \cite{C-D-K}, 
\cite{danna1}, \cite{DGM}, \cite{DGM2}, 
\cite{felix}, \cite{delgado}, \cite{garcia},
\cite{GMM}, \cite{GMM2}, \cite{waldi}.

Given a non-zerodivisor  $x \in \mathcal{O}$, set $\om=(\omega_1, \ldots, \omega_d)= v(x)$ and consider the good ideal $E= \om +S$. The set $\Ap(S,\om)= S \setminus E$ is called the Ap\'{e}ry set of $S$ with respect to $\om$. Often we will consider the case $\om  = \e$, and then we simply write $\Ap(S)= \Ap(S, \e)$. 
This set has useful applications in the study of the quotient ring $\mathcal{O}/(x)$. In the case of algebroid branches, $\om \in \N$ and $\Ap(S,\om)$ is a finite set of cardinality $\om$. Ap\'ery sets of numerical semigroups and they properties are very well-known. For an extensive treatment of numerical semigroups and semigroup rings the reader may consult the monography \cite{libroassidannapedro}. 
In the case $d \geq 2$, $\Ap(S,\om)$ is infinite, but it can be canonically partitioned in $N= \omega_1+\cdots+\omega_d$ sets, as proved in \cite[Theorem 4.4]{GMM}.

We recall the definition of this partition, which can be defined analogously for any set $A \subseteq S$ that is the complement of some proper good ideal. For this we need to recall several technical definitions that allow us to work combinatorially on a good semigroup.

Given a set $U \subseteq \lbrace 1, \ldots,d \rbrace$ and an element $\al \in \mathbb{N}^d$, we define the following sets:
\begin{eqnarray*}
    \Delta^S_U(\al)&=&\{\be\in S \hspace{0.1cm}|\hspace{0.1cm} \beta_i=\alpha_i \text{ for } i\in U \text{ and } \beta_j>\alpha_j \text{ for } j\notin U\}.\\
    \widetilde{\Delta}^S_U(\al)&=&\{\be\in S \hspace{0.1cm}|\hspace{0.1cm} \beta_i=\alpha_i \text{ for } i\in U \text{ and } \beta_j\geq\alpha_j \text{ for } j\notin U\}\setminus\{\al\}.\\
    \Delta^S_i(\al)&=&\{\be\in S\hspace{0.1cm}|\hspace{0.1cm} \beta_i=\alpha_i \mbox{ and } \beta_j>\alpha_j \mbox{ for } j\neq i\}.\\
    \Delta^S(\al)&=&\bigcup_{i=1}^d\Delta_i^S(\al).
\end{eqnarray*}
In particular, for $S=\N^d$, we set $\Delta_U(\al):=\Delta_U^{\N^d}(\al)$ and $\widetilde{\Delta}_U(\al):=\widetilde{\Delta}_U^{\N^d}(\al)$.
   %\[ \Delta_F(\al):=\Delta_F^{\N^d}(\al)=\{\be\in \N^d \hspace{0.1cm}|\hspace{0.1cm} \beta_i=\alpha_i \text{ for } i\in F \text{ and } \beta_j>\alpha_j \text{ for } j\notin F\}.\]
   In general, we denote by $\widehat{U}$ the set $\lbrace 1, \ldots, d \rbrace \setminus U$.

Given any subset $A \subseteq S$, we say that an element $\al\in A$ is a \it complete infimum \rm in $A$ if there exist $\be^{(1)},\ldots,\be^{(r)}\in A$, with $r \geq 2$, satisfying the following properties:
\begin{enumerate}
\item $\be^{(j)} \in \ds{F_j}{\al}$ for some non-empty set $F_j \subsetneq \{1,\ldots,d\}$.
\item For every distinct $j, k \in \lbrace 1, \ldots, r \rbrace$,  $\al = \be^{(j)} \wedge \be^{(k)} $. %(We note that this is equivalent to say $F_j\cup F_k=I$).
\item $\bigcap_{k=1}^r {F_k}= \emptyset$.
    \end{enumerate}
In this case we write $\al=\be^{(1)}\wt\be^{(2)}\cdots\wt\be^{(r)}$.

Furthermore, given $\al=(\alpha_1,\alpha_2,\ldots, \alpha_d)$ and $\be=(\beta_1,\beta_2,\ldots,\beta_d)$ in $\mathbb{N}^d$, we say that  $\al  \leq \leq  \be$ if and only if either $\al = \be$ or $\alpha_i<\beta_i$ for every $i\in\{1,\ldots, d\}$. In the second case we say that $\be$ \it dominates \rm $\al$ and use the notation $\al \ll \be$.

\medskip

The partition of $\Ap(S, \om)$ is defined in the following way.

\begin{definition} \rm \label{deflivelli}
Let $A= \Ap(S, \om)$. Set:
$$ B^{(1)}:=\{\al \in A : \al \  \mbox{is maximal with respect to} \le\le\},$$
$$C^{(1)}:= \{ \al \in B^{(1)} : \al=\be^{(1)}\wt\cdots\wt\be^{(r)} \mbox{ for } 1<r\leq d\mbox{ and } \bei{k} \in B^{(1)}\},$$
$$ D^{(1)}:=B^{(1)}\setminus C^{(1)}.$$
For $i>1$ assume that $D^{(1)},\dots , D^{(i-1)}$ have been defined and set inductively:
$$ B^{(i)}:=\{\al \in A \setminus (\bigcup_{j < i} D^{(j)}) : \al \  \mbox{is maximal with respect to} \le\le\},$$
$$C^{(i)}:= \{ \al \in B^{(i)} : \al=\be^{(1)}\wt\cdots\wt\be^{(r)} \mbox{ for } 1<r\leq d\mbox{ and } \bei{k} \in B^{(i)}\},$$
$$D^{(i)}:=B^{(i)}\setminus C^{(i)}.$$ 
By construction $D^{(i)}\cap D^{(j)}=\emptyset$, for any $i\neq j$ and, 
since the set $S \setminus A = \om +S$ has a conductor, there exists $N \in \mathbb N_+$ such that $A=\bigcup_{i=1}^N D^{(i)}$.
As in \cite{GMM} we prefer to enumerate the sets in this partition in increasing order
setting $A_i:= D^{(N-i)}$. Hence 
$A=\bigcup_{i=0}^{N-1} A_i.$ 
We call the sets $A_i$ the \it levels \rm of $A$. 
\end{definition}

Notice that in the previous works \cite{GMM}, \cite{GMM2} the levels are enumerated from $1$ to $N$. In this paper we prefer to shift them and start from $0$ in order to adapt our notation to the one in \cite{apery}.

%Given $\om \in S$, we can consider the good ideal $E=\om+S$. In this case its complement $A= S \setminus E=\Ap(S, \om)$ is the Ap\'ery set of $S$ with respect to $\om$. The main theorem of \cite{GMM} describes the number of levels of sets of the form $\Ap(S, \om)$. We have:

In \cite[Theorem 4.4]{GMM} it is proved that the number of levels of the Ap\'ery set $\Ap(S, \om)$ is equal to $\sum_{i=1}^d \omega_i$.

We recall that, if $\al, \be \in A$, $\al \ll \be$ and $\al \in A_i$, then $\be \in A_j$ for some $j>i$.  Moreover, the last level of the partition is $A_{N-1}=\Delta(\ga_E)=\Delta^S(\ga_E)$ (here $E=\om +S$).
If $S$ is local then $A_0=\{\boldsymbol {0}\}$. \\
Other basic properties of the Apéry set and its partition in levels are listed in \cite[Lemma 2.3]{GMM}.

In \cite{apery}, it is defined a slightly different partition in levels for the Ap\'ery set, only in the case of plane algebroid curves with two branches.
%The partition in levels of the Ap\'ery set defined in \cite{apery} in the case $d=2$ is slightly different from the one defined above. 
However, it is proved in \cite[Proposition 5.1]{GMM2} that in the case of Ap\'ery sets of plane algebroid curves the partition in \cite{apery} coincides with the one given in Definition \ref{deflivelli}. For this reason, since in this article we deal with plane curves, the results in \cite{apery} can be used as starting point of the inductive arguments in our work, even if we work with a partition in levels defined in a different way.

%\rc
%I think that the next Proposition (without proof, one can give [16] as reference)
%could help the reader to understand
%the behaviour of the Apery set (and in fact the semigroup) and its levels, in
%particular the role of subspaces.
%It could make it easier to understand the proof of 3.6

%This is a simple suggestion .... I make a reference to this proposition at some placeof the proof of Proposition ??
%\ec

In the introduction of \cite{apery} it is mentioned that all the results in that paper until Theorem 4.1 can be proved analogously for arbitrary $d \geq 2$. We discuss this fact more specifically, showing first a way to present a plane algebroid curve as a finite module over a power series ring in one variable. The following extends the content of \cite[Discussion, page 6]{apery} and is independent of the characteristic of the base field.

%In this section we present a discussion for an algebroid curve with $d$ branches very similar to that one for an algebroid curve with two branches in \cite{apery}. We write it here for the sake of completeness.

\begin{prop}
    \label{discussion}
    Let $\mathcal{O}=K[[X,Y]]/I$ be an algebroid plane curve with $d$ branches. %where $I=(H_1\cdot H_2\cdots H_d)$ and $H_1,\dots,H_d$ irreducible, be an algebroid curve with $d$ branches.
    Then, we can always write
    $$ \mathcal{O}=K[[x]]+K[[x]]y+K[[x]]y^2+\cdots+K[[x]]y^{e-1} $$
    where $v(x)=(e_1,\dots,e_d)=\min(v(\mathcal{O})\setminus\{(0,\dots,0)\})$, $e_1+\cdots+e_d=e$.
\end{prop}

\begin{proof}
    We can assume $I=(H_1\cdots H_d)$ with $H_1,\dots,H_d$ irreducible elements and pairwise coprime. Let us denote $\mathcal{O}$ also by $K[[x,y]]$, where $x=X+I$ and $y=Y+I$.
    If the $d$ branches defined by $H_1,\dots,H_d$ have all the same tangent, we can assume it is $Y=0$ and, according to Weierstrass' Preparation Theorem, we can assume that $H_j=Y^{e_j}+\sum_{i=0}^{e_j-1}a_i(X)Y^i$ where $e_j$ is the minimal power such that $H_j$ contains a pure power $a Y^{e_j}$, with $a\in K \setminus \{ 0 \}$, and $a_i(X)$ are all non-invertible power series in $K[[X]]$. Thus $H_1\cdots H_d=Y^{e}+\sum_{i=0}^{e-1}c_i(X)Y^i$ where $e=e_1+\cdots+e_d$ is the multiplicity of the curve and $c_i(X)$ are all non-invertible. 

If instead the tangents of the $d$-branches are not all the same, we can assume that at least one is $Y=0$ and, as above, $H_j=Y^{e_j}+\sum_{i=0}^{e_j-1}a_i(X)Y^i$ for each branch $H_j$ with tangent $Y=0$. Then, for each branch $H_k$ with tangent different from $Y=0$,  if we write it as $H_k(X+Y,Y)$ we get a term $Y^{e_k}$ where $e_k$ is the minimal degree of the nonzero terms of $H_k$. Hence, after applying the substitution $X=X+Y$  and Weierstrass' Preparation Theorem, we get again $H_1\cdots H_d=Y^{e}+\sum_{i=0}^{e-1}c_i(X)Y^i$ where $e=e_1+\cdots+e_d$ is the multiplicity of the curve and $c_i(X)$ are all non-invertible. 

It is clear that, in both cases, we can express $\mathcal{O}$ as a $K[[x]]$-module minimally generated by $1,y,y^2,\dots,y^{e-1}$, with $v(x)=(e_1,\dots,e_d)$ and $e_1+\cdots+e_d=e$.
\end{proof}

\begin{oss}\label{discussion2}
Let us keep the same notations of the previous proposition.
  Let $F,G\in\mathcal{O}$ be two elements such that $\mathcal{O}$ is a $K[[F]]$-module minimally ge\-nerated by $1,G,G^2,\dots,G^{N-1}$, with $N=n_1+\cdots+n_d$ and $v(F)=(n_1,\dots,n_d)$. Hence $\mathcal{O}\cong K[[X,Y]]/(\Phi)$, where $\Phi(X,Y)=Y^N+\sum_{i=0}^{N-1}b_i(X)Y^i$  
 comes from the relation of dependence of $G$ over $K[[F]]$ in degree $N$. Indeed, there is a surjective homomorphism $\varphi:K[[X,Y]]\rightarrow\mathcal{O}$, mapping $X$ to $F$ and $Y$ to $G$, whose kernel contains $(\Phi)$.  Now,   since $K[[X,Y]]$ is a $2$-dimensional UFD, $\ker\varphi$ has to be intersection of $d$ principal prime ideals $P_1,\dots,P_d$, hence $P_i=(H_i)$ and $\ker\varphi=(H_1\cdots H_d)$. Moreover, $H_1\cdots H_d$ divides $\Phi$, so it has to be of the form $Y^j+\psi(X,Y)$, with $j\le N$ and, since $\mathcal{O}$ is minimally generated by $1,G,G^2,\dots,G^{N-1}$ as $K[[F]]$-module, then $j=N$ and $(H_1\cdots H_d)=(\Phi)$. \ec
 
 Notice that the classes $x=X+I, y=Y+I\in\mathcal{O}$ always satisfy the condition requested for $F$ and $G$. Hence, by Proposition \ref{discussion}, we can always assume that $\mathcal{O}=K[[x]]+K[[x]]y+K[[x]]y^2+\cdots+K[[x]]y^{e-1}$, where $v(x)=(e_1,\dots,e_d)=\min(v(\mathcal{O}\setminus\{(0,\dots,0)\})$, $e_1+\cdots+e_d=e$.
 Moreover, up to replacing $y$ with $y+\alpha x$ (with $\alpha \in K$), we can choose $y$ in such a way that $v(y)=(r_1,\dots,r_d)$ with $r_i>e_i$ for those indices $i$ such that $H_i$ has tangent $Y=0$ and $r_j\ge e_j$ for the remaining indexes.
\end{oss}

 As consequences of Propostion \ref{discussion}, we can state the two following results generalizing (with the same identical proofs) Proposition 3.8 and Theorem 4.1 of \cite{apery}.
 
 Let $\mathcal{O}=K[[x]]+K[[x]]y+K[[x]]y^2+\cdots+K[[x]]y^{e-1}$ be a plane curve expressed as in Proposition \ref{discussion}.  The element $\e=(e_1, \ldots, e_d)$ is as usual the minimal element of $v(\mathcal{O})$ having all components distinct from zero. 
 Set $R_0= K$ and for $i = 1, \ldots, e-1$, 
$$R_i = K[[x]]+K[[x]]y + \cdots + K[[x]]y^{i}.$$
Similarly set $T_0 = K$ and for $i = 1, \ldots, e-1$
$$
T_i =  \left\lbrace y^{i} + \phi \, | \, \phi \in R_{i-1} \mbox{ and } v(y^i + \phi) \not \in  v(R_{i-1}) \ec  \right\rbrace. $$

 \begin{prop}
     \label{3.8apery}
Let $A_i$ denote the levels of $\Ap(v(\mathcal{O}))$. Then for $i=0, \ldots, e-1,$
$ A_i = v(T_i)$.
 \end{prop}

 \begin{theorem}
     \label{4.1apery}
     Let $\mathcal{B}(\mathcal O)$ denote the blow-up of $\mathcal{O}$ and suppose $\mathcal{B}(\mathcal O)$ to be also local. Let $A'_i$ denote the levels of $\Ap(v(\mathcal{B}(\mathcal O)), \e)$. Then, for $i=0, \ldots, e-1,$ one has $A'_i= A_i -i\e$.
 \end{theorem}

The aim of the next sections is to extend Theorem \ref{4.1apery} to the case where the blow-up of $\mathcal{O}$ is not local. In this case it is no more true that $\mathcal B(\mathcal O)$ can be presented as a quotient of $K[[X,Y]]$, so we cannot apply Proposition \ref{discussion} and Remark \ref{discussion2}. To proceed in this direction, we will need to consider the levels of the Ap\'ery set of non-local good semigroups.

\section{Preliminary results on good semigroups}

In this section we prove several technical results on good semigroups that will be needed in Section 4. The proofs often require the combinatoric methods developed in the previous works \cite{GMM}, \cite{GMM2}.
%In the following we number the levels of the complement $A$ of a good ideal starting from $A_0$, rather than from $A_1$.
We start by recalling
the main result of Section 4 of \cite{GMM2}, restated with the new notation, renumbering the levels of the Ap\'ery set (or more in general of the complement of a good ideal) starting from $0$ rather than from $1$. 

Along the section $S \subseteq \N^d$ will denote an arbitrary good semigroup (not necessarily local) and 
$A = S \setminus E =\bigcup_{i=0}^{N-1} A_i$ the complement of a good ideal $E$, partitioned in levels as in Definition \ref{deflivelli}. If $S$ is numerical,
$A= \{w_0, \ldots, w_{N-1} \}$ is finite and we set $A_i= \{w_i\}$.

%Let $T$ be any good semigroup (not necessarily local)
%and let $A= \bigcup_{i=0}^{N-1} A_i$ be %its Apery set with %respect to some nonzero element.
%the complement of a good ideal, partitioned in levels as in 
%Definition \ref{deflivelli}. If $T$ is numerical,
%$A= \{w_0, \ldots, w_{N-1} \}$ is finite and we set $A_i= \{w_i\}$.

We define a \it level function \rm
$\lambda: S \to \{0, \ldots, N\}$ in the following way:
\begin{itemize}
\item If $\al \in A_i$, $\lambda(\al)= i$.
\item If $\al \not \in A$, %and there exists $\beta \in A$ such that $\be \geq \al$, then 
$\lambda(\al)= 1+ \max\{i \mbox{ such that } \al > \te \mbox{ for some } \te \in A_i \}$.
\end{itemize}
%\end{definition}

\begin{theorem}
\label{thmnonlocallevels} \rm \cite[Theorem 4.5]{GMM2} \it
Let $S= S_1 \times S_2 $ be a direct product of two arbitrary good semigroups. Let $E \subsetneq S$ be a good ideal and set $A:=S \setminus E$. Then, 
given $\al = (\al^{(1)}, \al^{(2)}) \in A$ $(\al^{(i)} \in S_i$, for $i=1,2$), the level of $\al$ in $A$ is equal to
$$ \lambda(\al^{(1)}) + \lambda(\al^{(2)}).  $$
\end{theorem}

We recall that two elements $\al, \be \in S$ are \it consecutive \rm if $\al < \be$ and there are no elements $\de \in S$ such that $\al < \de < \be$.
The function $\lambda$ has following property.

\begin{lemma}
    \label{lambda}
    Let $S$ be any good semigroup and let $\al \in S$. Let $E \subsetneq S$ be a good ideal and set $A:=S \setminus E$. Then for $j < N$, $\lambda(\al) \leq j$ if and only if there exists $\be \in A_j$ such that $\al \leq \be$. In particular, if $\te \in S$ and  $\al \leq \te$, then $\lambda(\al) \leq \lambda(\te)$.
\end{lemma}

\begin{proof}
    If $\al \in A$ this is straightforward. Suppose $\al \in E$ and set $\lambda(\al)=h$. Let $\te \in A$ be a maximal element such that $\te < \al$. By definition of $\lambda$, $\te \in A_{h-1}$. Now, if there exists $\be \in A_j$ such that $\al \leq \be$, it follows that $j \geq h-1$. If $j \geq h$ we are done. If $j= h-1$, by \cite[Lemma 2.8]{GMM2} we get $\al \in A_{h-1}$ and this is a contradiction. 

    Conversely, if $\lambda(\al)=h=N$, then clearly $\te \in A_{N-1}$ and there are no elements of $A$ larger than or equal to $\al$.
    Thus we suppose $h < N$ and prove that we can find $\be \in A_h$ such that $\al \leq \be$. Clearly no elements of $A_h$ are smaller than $\al$.
%there exists $\eps \in A_{i-1}$ such that $\eps < \al$ and no elements of $A_i$ are smaller than $\al$. We can assume $\eps$ to be a maximal element in $A_{i-1}$ smaller than $\al$.
%We first consider the case where $S$ is local. By \cite[Lemma 3.3]{apery}, \cite[Proposition 5.1]{GMM2}, we know that every element of $A_i$ is dominated by some element in $A_{i+1}$. Thus it is sufficient to show that there exists $\be \in A_i$, such that $\al < \be$. 
%We can find $\be \in A_i$ such that $\be \gg \eps$. By assumption $\be $ cannot be smaller than $\al$, and if $\be > \al$ we are done. Therefore suppose we are in the situation where $\al \wedge \be < \al$ and $\al \wedge \be < \be$. Clearly $\al \wedge \be > \eps$, thus by maximality of $\eps$ and by the assumption $\lambda(\al)=i$, we must have $\al \wedge \be \in E$. We can now possibly replace $\be$ with an element in $A_i$ (that we call again $\be$) such that $\al \wedge \be \leq \al$ is the maximal possible. By way of contradiction say that $\al \wedge \be < \al$. Since still $\al \wedge \be > \eps$, we have $\de:= \al \wedge \be \in E$. We fix now coordinates saying that $\al \in \Delta^S_F(\de)$ and $\be \in \Delta^S_G(\de)$ with $G \supseteq \widehat{F}$.
Let $\be \in A_h$ be such that the element $\de= \al \wedge \be$ is the maximal possible. If $\de = \al $ we are done, hence suppose by way of contradiction that $\de < \al$. By the assumption $\lambda(\al)=h$, we also have $\de < \be$.
We can fix coordinates saying that $\al \in \Delta^S_U(\de)$ and $\be \in \Delta^S_V(\de)$ with $V \supseteq \widehat{U}$.
We need to produce an element $\te \in A_h$ such that $\te \wedge \al > \de$. We can do it proceeding exactly as in Case 1 and Case 2 of the proof of \cite[Proposition 2.10]{GMM2}, noticing that $\al \in E$ and therefore if $\de$ and $\be$ are consecutive, we cannot have $\de \in A$ by \cite[Theorem 2.7]{GMM2}.
 (for convenience of the reader we are adopting here the same notation of that proof, except for the fact that the index of the level of $\be$ is shifted by one). Since in this way we find a contradiction, we must have $\al \wedge \be = \al$ and $\be > \al$. 
  %After the local case is established, let us assume $S$ to be not local. We can restrict to assume $S= S_1 \times S_2$ where $S_1$ and $S_2$ are not necessarily local but, by induction on the number of branches, we suppose that the thesis of the theorem holds for both $S_1$ and $S_2$.
\end{proof}

The next lemma proves the existence of ascending sequences of elements, one for each level, satisfying some extra condition on their respective positions. 

\begin{lemma}
    \label{lemmanodominore}
   Let $S $ be an arbitrary good semigroup. Let $E \subsetneq S$ be a good ideal and set $A:=S \setminus E$. Then for every $i > j \geq 0$ and $\al \in A_i$, there exists $\be \in A_{j}$ such that $\be < \al$ and, %either $\be \ll al$ or, 
   if $\al \in \Delta^S_U(\be)$ then $\widetilde{\Delta}^S_{\widehat{U}}(\be) \subseteq A$.
\end{lemma}

\begin{proof}
Observe that if there exists $\be \in A_j$ such that $\al \gg \be$, the thesis is satisfied since $U = \emptyset$ and $\widetilde{\Delta}^S_{\widehat{U}}(\be) = \lbrace \be \rbrace \subseteq A$.
    %Thus we can work by induction on $i$ since for $i=1$ the thesis follows choosing $\be = \boldsymbol{0}$.
    First let us consider the case $j=i-1$. This case will also provide a base for an induction on $i$.
    By \cite[Proposition 2.10]{GMM2} there exists $\be \in A_{i-1}$ such that $\be < \al$. We can assume that there are no other elements in $A_{i-1}$ between $\al$ and $\be$. Let $\te \in S$ be an element consecutive to $\be$ such that $\be < \te \leq \al$. Hence, $\te \in \Delta^S_H(\be)$ with $H \supseteq U$ and $\widetilde{\Delta}^S_{\widehat{U}}(\be) \subseteq \widetilde{\Delta}^S_{\widehat{H}}(\be) $.
    If by way of contradiction $\widetilde{\Delta}^S_{\widehat{U}}(\be) \nsubseteq A$, by \cite[Theorem 2.8]{GMM} the element $\te \in A_{i-1}$. In particular, $\te < \al$ and this contradicts the fact that no elements between $\al$ and $\be$ are in $A_{i-1}$.
    %re exists an element $\te \in \widetilde{\Delta}^S_{F}(\be) \cap A_{i-1}$ such that $\be$ and $\te$ are consecutive. For some index $k $ such that $\theta_k > \beta_k$ we have $k \not \in F$ and therefore also $\alpha_k > \beta_k$. This implies that $\be < \te \wedge \al < \te$ and $\be < \te \wedge \al < \al$. The first chain of inequalities implies that $\te \wedge \al \in A_{i-1}$ by \cite[Lemma 2.8]{GMM2}. But this contradicts our assumption about not having elements in $A_{i-1}$ between $\al$ and $\be$. Hence, we must have $\widetilde{\Delta}^S_{\widehat{F}}(\be) \subseteq A$. 
    
    By induction, after finding $\be \in A_{i-1}$ satisfying the thesis, taking $j < i-1$, we can find $\de \in A_j$ such that $\be \in \Delta^S_V(\de)$ and $\widetilde{\Delta}^S_{\widehat{V}}(\de) \subseteq A$. It follows that $\al \in \Delta^S_H(\de)$ with $H \subseteq U \cap V$. Since $\widehat{H} \supseteq \widehat{U} \cup \widehat{V} \supseteq \widehat{V}$, we get $ \widetilde{\Delta}^S_{\widehat{H}}(\de) \subseteq \widetilde{\Delta}^S_{\widehat{V}}(\de) \subseteq A. $ This concludes the proof.
\end{proof}

\begin{oss}
    \label{chainofnodes}
    The proof of Lemma \ref{lemmanodominore} shows that, starting from an element $\al^{(N-1)} \in A_{N-1}$ we can find a chain of elements
    $$ \boldsymbol{0}= \al^{(0)} < \al^{(1)} < \cdots < \al^{(j)} < \cdots < \al^{(N-2)} < \al^{(N-1)}, $$ such that for every $j=0, \ldots, N-1$, $\al^{(i)} \in A_i$ and, for every $k < j$, if $\al^{(j)} \in \Delta^S_U(\al^{(k)})$ for some $U \neq \emptyset$, then $\widetilde{\Delta}^S_{\widehat{U}}(\al^{(j)}) \subseteq A$.
\end{oss}

%\subsection{Results involving subspaces}

%\bc Suggestion: say that what we write next is about subspaces and the proofs are technical and later we need only the statement of  Lemma \ref{congruencelemma}. \

All the results from now until the end of the section are very technical and use the notion of subspaces of a good semigroup introduced in \cite{GMM}. The only result needed in the next sections is the statement of Lemma \ref{congruencelemma}. \\

Let $S \subseteq \N^d$ be an arbitrary good semigroup and let $A= \bigcup_{i=0}^{N-1}$ be its Ap\'ery set with respect to a nonzero element $\om= (\omega_1, \ldots, \omega_d)$. Set as usual $E= S \setminus A$ and denote the conductor of $E$ by $\boldsymbol{c_E}=(c_1, \ldots, c_d)= \g + \om + \boldsymbol{1}$. %Denote the set of coordinates of $\N^d$ by $I=\{ 1, \ldots, d\}$. 

The following definition and properties are taken from \cite[Section 3]{GMM}.

We recall the next useful fact which describes the behavior of the levels of the Ap\'ery set for large elements.

\begin{prop}\label{newprop}
\rm \cite[Proposition 2.9]{GMM} \it
Let $\C$ be the conductor of $E=\om +S$, let $\de \ge \C$ and
let $\al \in \N^d$ be such that $\al \not<\de$ and $\te = \al \wedge \de$.
Let $U=\{i : \alpha_i < \delta_i\}$. 
%$V = \{i : \alpha_i \ge \delta_i\}$. 
Then the
following conditions are equivalent:
\begin{enumerate}
\item
$\al \in A_j$
\item
$\widetilde{\Delta}_{U}(\al)\cup \{\al\}\subset A_j$
\item
$\widetilde{\Delta}_{U}(\te)\cup \{\te\}\subset A_j$
\end{enumerate}
In particular, as a consequence, if $\de = \C$
the Ap\'ery set $A=\Ap(S,\om)$ and its levels $A_j$ depend only on
the finite subset $ \{\al \in A : \al \le \C\}$.
\end{prop}

\begin{definition}
\label{subspacedef}
\rm
Pick a non-empty set $ U\subseteq \{ 1, \ldots, d\}$. %and a good ideal $E \subseteq S$. Set $A=S\setminus E$ and let $\bs{c_E}=(c_1,\ldots,c_n)$ be the conductor of $E$. 
For $\al \in \N^d$ such that $\alpha_j=c_j$ for all $j\in \wh U$, define $$\al(U)= \widetilde{\Delta}_U(\al)\cup \{\al\}$$ 
%if $U \subsetneq I$ and $\al(U):= \al$ if $U=I$. 
We say that $\al(U)$ is an \it $U$-subspace \rm (or simply a subspace) of $\N^d$. %As consequence of Propositions \ref{infsem} and \ref{infap}, it follows:
%\mc Why not $\al(U)= \widetilde{\Delta}_U(\al)\cup \{\al\}$ directly, this avoid the
%distinction of $I=U$ or $I\neq U$.\ec
%\bc It seems logical, it should work but I should check carefully if it changes the form properties that involves subspaces, minimum and deltas. \ec
We have that:
\begin{itemize}
    \item If $\al \in E$, then $\al(U) \subseteq E$ and in this case we say that it is an $U$-subspace of $E$, or that $\al(U) \in E(U)$.
    \item If $\al \in A$, then $\al(U) \subseteq A$ and in this case we say that it is an $U$-subspace of $A$. In particular, if $\al \in A_i$, the subspace $\al(U) \subseteq A_i$ and we write shortly that $\al(U) \in A_i(U)$.
\end{itemize}
\end{definition}
%We notice that, by taking $E:=S$, we define also for every $U$ the $U$-subspaces of the good semigroup itself. In the following we omit the $U$ and we simply say \it subspace \rm when the set $U$ is clear in the contest or from the notation.
%We observe the following fact: 
%\begin{oss}
%\label{remsubcont}
Observe that if
$\de(V)$ is a subspace and $U\supseteq V$. If $\al\in \widetilde{\Delta}_{V}(\de)$,
%\mc why $\al \in S$? there is no condition on $\de$ \ec, 
then $\al(U)\subseteq \de(V)$.
%\mc If $\alpha_j<\delta_j$ for some $j\in U\setminus V$ (the definition imposes
%$\alpha_i=\delta_i$ for $i\in V$ only) and
%$\be\in \widetilde{\Delta}_U(\al)$ is such that $\beta_j=\alpha_j$ ($j\in
%U\setminus V$), then $\beta_j < \delta_j$, so $\be\in \al(U)$ and
%$\be\notin \de(V)$. Thus $\al(U)\subset \de(V)$ is not true.
%Obviously is true if $\al \in \widetilde{\Delta}_{V}(\de)$.
%\ec 

The dimension of a subspace is defined accordingly to its intuitive geometric representation. We say that $\al(U)$  has dimension equal to the cardinality of $\wh U$. Indeed, the subspaces of dimension zero are points, those of dimension one are lines, those of dimension two are planes, and so on.

\begin{figure}[H]
\begin{subfigure}{.5\textwidth}
\centering
 \includegraphics{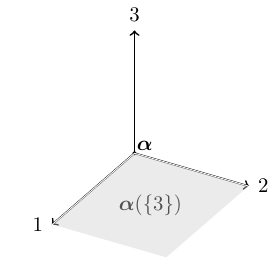}
%\caption{\footnotesize{Here is represented the plane $\alpha(\{3\})$} which is a subspace of dimension 2.}
\end{subfigure}%
\begin{subfigure}{.5\textwidth}
\centering
\includegraphics{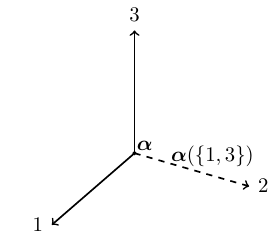}
%\caption{\footnotesize{The dashed line $\alpha(\{1,3\})$} represents  a subspace of dimension 1.}
\end{subfigure}
\caption{\footnotesize{In the figure of the left is represented the plane $\alpha(\{3\})$ which is a subspace of dimension 2. In the figure on the right, the dashed line $\alpha(\{1,3\})$ represents a subspace of dimension 1.}}
\label{fig:sub}

\end{figure}

%In the following, if there are not ambiguity on the case that we are treating, we write simply $\al(U)$.\\
%In order to agree with the geometric representation of these subspaces, we define the dimension of a subspace $\al(U)$ as the cardinality $|\wh U|= d-|U|$. The subspaces of dimension zero correspond to the single elements of $S$. 
%Then we call: \emph{points} the subspaces of dimension 0, \emph{lines} the subspaces of dimension 1, \emph{planes} the subspaces of dimension 2, and \emph{hyperplanes} the subspaces of dimension $d-1$.

 The proof of the following lemma is based on part of the argument used to prove \cite[Theorem 4.4]{GMM}.

\begin{lemma}
\label{subspaceslemma} 
Fix an index $i \in \{ 1, \ldots, d \}$.
Let $V$ be a nonempty set of indexes not containing $i$ and set $W:=V \cup \lbrace i \rbrace$. Choose a
subspace of the form $\te(V)$ contained in $A$ such that $\te$ is a minimal element for which a subspace of $A$ of such form exists. 
Then, 
there exist $\omega_i$ distinct subspaces of the form $\be^{(0)}(W), \ldots, \be^{(\omega_i-1)}(W) \subseteq \bigcup_{l < \lambda(\te)}A_l$ such that the coordinates $\be^{(0)}_i, \ldots, \be^{(\omega_i-1)}_i$ form a complete system of residues modulo $\omega_i$.
\end{lemma}

To help the reader we add separately the proof in the case $d=2$, and then the proof of the general case.

\begin{proof} (Proof of Lemma \ref{subspaceslemma} in the case $d=2$) \\
First set $i=1$. Clearly by definition of the conductor $\boldsymbol{c_E}=(c_1, c_2)$ of the good ideal $E$ there are infinitely many elements $\al \in E$ such that $\alpha_2= c_2$.
%sets of the form $\Delta_1(\al)$ contained in $E$ (it is sufficient to pick any $\al \gg \g + \om$). 
Thus, for every $j=0, \ldots, \omega_1-1$ we can find a unique minimal element $\al^{(j)} \in E$ such that $\alpha^{(j)}_1 \equiv j \mbox{ mod } \omega_1$ and $\alpha^{(j)}_2= c_2.  $ %By the properties (G1)-(G2)-(G3) of the good ideal $E$ it follows that $\Delta_1(\al^{(j)}) \subseteq E$.
Hence for every $j$ there exists $n_j \geq 1$ such that $\al^{(j)} - n_j \om \in A$. For $\gamma \in S$, set $H_1(\gamma)= \lbrace   \de \in S | \delta_1 = \gamma_1 \rbrace$. If $H_1(\al^{(j)}-n_j \om) \cap E \neq \emptyset$ we can continue subtracting multiples of $\om$ to some element in $H_1(\al^{(j)}) \cap E$ until we find an element $\be^{(j)} \in A$ such that $\beta^{(j)}_1 \equiv \alpha^{(j)}_1 \equiv j \mbox{ mod } \omega_1$ and $H_1(\be^{(j)}) \subseteq A$. Without loss of generality we can assume  
$\be^{(j)}$ to be the minimal element of $H_1(\be^{(j)})$. Now let $\te \in A$ be the minimal element of $S$ such that $\theta_1 = c_1$ and
$\Delta_2(\te) \subseteq A$. We show that $\lambda(\be^{(j)}) < \lambda(\te)$ for every $j$. Indeed, by minimality of $\be^{(j)}$ in $H_1(\be^{(j)})$, using property (G1), we must have $\beta^{(j)}_2 \leq \theta_2$ and by construction of $\be^{(j)}$ we must have $\beta^{(j)}_1 < c_1 = \theta_1 $. Using that $\Delta^S_1(\be^{(j)}) \subseteq  H_1(\be^{(j)}) \subseteq A$ we get the inequality $\lambda(\be^{(j)}) < \lambda(\te)$ by \cite[Lemma 2(3)]{DGM} together with the definition of levels. \end{proof}

\begin{proof} (Proof of Lemma \ref{subspaceslemma} for arbitrary $d$) \\
Relabelling the indexes, we can assume that $W=\lbrace 1,\ldots, i \rbrace$ and $V=\lbrace 1,\ldots, i-1 \rbrace$. Denoting by $l=\lambda(\te)$, we have that $\te(V)\subseteq A_l$, hence it is clear that there exist infinitely many $W$-subspaces contained in level $A_l$ (a space of dimension $j$ contains infinitely many spaces of dimension $j-1$).
Among them, for every $j=1,\ldots, w_i$, there exist subspaces
$\te^j(W)\in A_l(W)$ minimal with respect to the property of having $\te^j_i \equiv j $ mod $w_i$.

For each $j$, we show that  $\tdE{i}{\te^j(W)} \neq \emptyset$. Indeed, after fixing $\te^j(W)$, using the fact that there are infinitely many $W$-subspaces contained in $\te(V)$, we can find $\te^{\prime}(W) \in A_l(W)$ such that  $\theta^{\prime}_i > \theta^j_i$ (observe that since they are in the same level necessarily $\theta^{\prime}_h = \theta^j_h$ for some $h < i$). % If one dominates the other one they cannot stay in the same level
Now, if we assume $\tdE{i}{\te^j(W)} = \emptyset$, applying \cite[Theorem 3.7]{GMM} to $\te^j(W)$ and $\te^{\prime}(W)$, we can write
$$ \te^j(W) = \te^{\prime}(W) \wt \al^1(W) \wt \cdots \wt \al^r(W) $$
where $\al^m(W) \in \td{i}{\te^j(W)} \subseteq A(W)$ and we may assume $\al^m(W)$ to be consecutive to $\te^j(W)$ for all $m\in{1,\ldots r}$. By \cite[Theorem 3.9.1]{GMM}, for every $m$, $\al^m(W) \in A_j(W)$ implies that $\te^j(W)$ has to be in a lower level. This is a contradiction (for a graphical representation see Figure \ref{fig:maintheorema}). %(\textcolor{blue}{If you look the figure a) from above is like apply \cite[Proposition 2.4]{GMM} in two branches,"theorem of white elements")}.

Hence, we can set $\bs{\tau}^j(W)$ to be a minimal element in $\tdE{i}{\te^j(W)}$. We define $\overline{\om}$ such that $\overline{\omega_k}=\omega_k$ if $k\in W$ and $\overline{\omega_k}=c_k$ otherwise, and, starting from $\bs{\tau}^j(W)$ and subtracting multiples of $\overline{\om}(W)$, we find a unique $m_j \geq 1$ such that $  \bs{\tau}^j(W) - m_j \overline{\om}(W)=:\be^j(W) \in A(W)$ (see Figure \ref{fig:maintheoremb}).
%\textcolor{blue}{If we consider a subspace in E and we start to subtract the multiplicity we obtain a subspace in the Apery set, it is clear if you take an element in subspace which has the infinite components sufficiently large} .

Consider now %for $k \in W$ the set $H_k(\al^j(W))= \{ \}$ defined in Definition \cite[Definition 4.2]{GMM}. 
the set $H_i(\al^j(W))= \{\be(U) \subseteq S | \beta_i = \alpha_i^j \}$.
In the case this set contains some subspace of $E$, starting by one of these subspaces and subtracting multiples of $\overline{\om}(U)$, we can repeat the process and, after changing names, we can finally assume to have a collection of subspaces $\be^1(W), \ldots, \be^{w_k}(W) \in A(W)$ such that for every $j$, $\beta^j_i \equiv \theta^j_i \equiv j $ mod $w_i$ and $H_i(\be^j(W)) \subseteq A(W)$. We can further replace $\be^j(W)$ by another subspace, and assume that $\be^j(W)$ is the minimal $W$-subspace in the set $H_i(\be^j(W))$ (this minimal subspace is well-defined by property (G1), see the results in \cite[Sections 3 and 4]{GMM}).
%If we denote by $X^j$ the set of subspaces $H_k(\be^j(W))$ which satisfies the previous condition. We can define $$\be^j(W)=\bigwedge_{H_k(\be^j(W))\in X^j} 
%H_k(\be^j(U))$$ 
%\textcolor{red}{I used wedge since I mean that we have to consider the minimum of these elements pairwise components and not the infimum as minimum of the elements of a set} which is well-defined by Proposition \cite[Proposition 3.5.1]{GMM} .
%After this change, it follows that $\be^1(W), \ldots, \be^{w_k}(W)$ are totally ordered with respect to the standard order $\leq$ and, by relabeling the indexes, suppose $$\be^1(W) > \cdots > \be^{w_k}(W).$$

  To conclude, notice that for every $j$, the level of $\be^j(W)$ has to be strictly lower than $l$ since $\te^j(U)$ has been chosen to be the minimal in $A_l$ having $k$-th component congruent to $j$ modulo $w_j$.
  \end{proof}

\begin{figure}[H]
\begin{subfigure}{.53\textwidth}
 \includegraphics[scale=1.6]{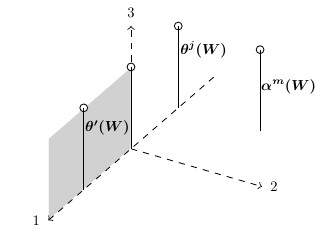}
  \caption{}
  \label{fig:maintheorema}
\end{subfigure}%
\begin{subfigure}{.5\textwidth}
 \includegraphics[scale=0.8]{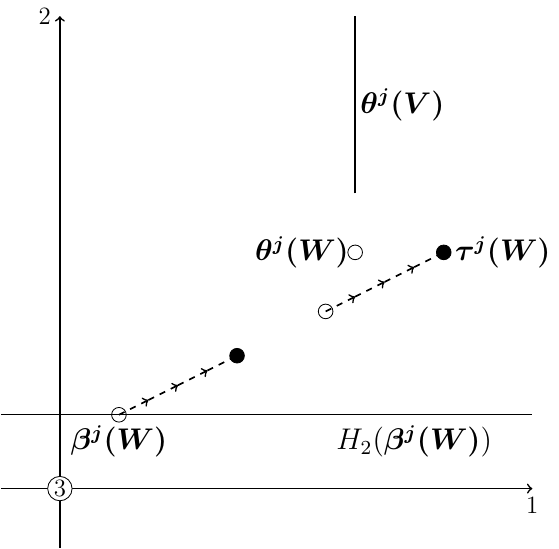}
   \caption{}
   \label{fig:maintheoremb}
\end{subfigure}%
\caption{\footnotesize{\ref{fig:maintheorema}: we have $d=3$, $U=\{1,2\}$, $\te^j(W),\te^{\prime}(W),\al^m(W)$ are lines. \ref{fig:maintheoremb}: this is a perspective from "above" of the case $d=3$, $U=\{1,2\}$, $V=\{1\}$. In this case $\te^{j}(V)$ is a plane contained in $A$; $\te^j(W),\bs{\tau}^j(W),\be^j(W)$ are lines.}}
\label{fig:maintheorem}
\end{figure}
%\begin{comment}
%\end{comment}  
%\end{proof}

\begin{lemma}
    \label{congruencelemma}
  Let $S$ and $A$ be defined as above.
   %be an arbitrary good semigroup and let $A= \bigcup_{i=0}^{N-1}$ be the Ap\'ery set with respect to a nonzero element $\om= (\omega_1, \ldots, \omega_d)$. Set as usual $E= S \setminus A$. 
  Then, it is possible to find a sufficiently large element $\he \gg \g + \om$ such that, given any index $i$ and any element $\al \in A$ such that $\alpha_i \geq \eta_i$, there exists $\de \in \Delta^E_i(\al)$ such that $\de = m \om + \be$, with $m \geq 1$, $\be \in A$, and $\lambda(\be) < \lambda(\al)$.
\end{lemma}

\begin{proof}   
  %Since $\al \in A$, then $\alpha_j \leq \gamma_j + \omega_j$ for some $j$. 
  Fixed a coordinate $i$, we want to find an element $\he(i) \gg \g + \om$ such that if $\alpha_i \geq \eta(i)_i$, then there exists $\de \in \Delta^E_i(\al)$ of the required form. Then we can simply define $\he$ as the minimal element of $S$ that is larger or equal than 
  %sup of
  all the elements $\he(1), \ldots, \he(d)$ with respect to the partial ordering $\leq$. %\bc I use the notation with subspaces from the paper \cite{GMM}. The set $\te(V)= \widetilde{\Delta}_V(\te)$ where $\te \leq \g + \om + \boldsymbol{1}$ and $k \in \widehat{V} $ if and only if $\theta_k = \gamma_k + \omega_k +1 $. \ec

  Let $V$ be a nonempty set of indexes not containing $i$ and set $W:=V \cup \lbrace i \rbrace$. Given the minimal subspace of the form $\te(V)$ contained in $A$, %(say $\te(V) \subseteq A_j$), 
  by Lemma \ref{subspaceslemma} %the same argument of the proof of \cite[Theorem 4.4]{GMM} (\rc Nicola is rewriting this part to check if it works, we can decide if include it \ec) 
  we can find $\omega_i$ distinct subspaces of the form $\be^{(0)}(W), \ldots, \be^{(\omega_i-1)}(W) \subseteq \bigcup_{l < \lambda(\te)}A_l$ such that the coordinates $\be^{(0)}_i, \ldots, \be^{(\omega_i-1)}_i$ form a complete system of residues modulo $\omega_i$. For every $j=0, \ldots, \omega_i-1$, define $\boldsymbol{\tau}^{(j)}:= \be^{(j)} + m_j \om$ where $m_j$ is the minimal positive integer such that $\be^{(j)} + m_j \om \gg \g + \om$. Then set $\he(V)$ equal to the element $\boldsymbol{\tau}^{(j)}$ which has the largest $i$-coordinate. Finally, set $\he(i)$ to be the minimal element of $S$ larger or equal than all the elements $\he(V)$ for every $V$ not containing $i$.
  
  Now we can pick $\al \in A$ and suppose that $\alpha_i \geq \eta(i)_i$. Since $\al$ has at least one coordinate larger than the conductor, it belongs to an infinite subspace of $A$ of the form $\te'(V)$ with $i \not \in V$. In particular $V$ is nonempty and $\alpha_k \leq \gamma_k + \omega_k$ for all $k \in V$.
  Fixing this set $V$, we can take the elements $\be^{(j)}$ and $\boldsymbol{\tau}^{(j)}$ defined previously. Clearly $ \alpha_i \equiv \beta_i^{(j)}$ modulo $\omega_i$ for some $j$. Hence there exists $m \geq 1$ such that $$ \alpha_i = \beta_i^{(j)} + m \omega_i \geq \eta(i)_i \geq \tau^{(j)}_i = \beta_i^{(j)} + m_j \omega_i.$$
  
  Set $\de:= \be^{(j)}  + \boldsymbol{\epsilon} + m \om$ where $\boldsymbol{\epsilon}$ is an element of $\N^d$ such that $\epsilon_k = 0$ for $k \in V \cup \lbrace i \rbrace$, and $\epsilon_k > \alpha_k$ for the other coordinates.
  Notice that with these assumptions $ \be^{(j)} + \boldsymbol{\epsilon} \in  \be^{(j)}(W) \subseteq A$ and $\de \in S$ since it is larger than the conductor (notice that $m \geq m_j$). Observe that $\delta_i = \alpha_i$ and, since a subspace is all contained in the same level, 
  %(cf. \cite[Proposition 2.9]{GMM}), 
  observe also that $\lambda(\be^{(j)} + \boldsymbol{\epsilon})= \lambda(\be^{(j)}) < \lambda(\te) \leq \lambda(\te')=\lambda(\al)$. Furthermore, for $k \in V$ we have $\delta_k > \gamma_k + \omega_k \geq \alpha_k$ and for $k \not \in W$ we have $\delta_k > \alpha_k$ by definition of $\boldsymbol{\epsilon}$. In conclusion we obtain $\de \in \Delta^E_i(\al)$.
\end{proof}

\section{Semilocal rings associated to plane curves}

In this section we extend \cite[Theorem 4.1]{apery}  to the case where the blow-up of the coordinate ring of a plane curve is not local. In the first part of the section we describe the level of the Ap\'ery set of the value semigroup of a semilocal ring $R$ as sets of values of specific subsets of $R$. In the second part we describe how the levels of the Ap\'ery set of the value semigroup behave when passing from the ring of a plane curve to its blow-up and vice-versa.

\subsection{The Ap\'ery set of the value semigroup of a semilocal ring}

Let $R \cong \mathcal{O}_{1} \times \cdots \times \mathcal{O}_c $ be a direct product of local rings $\mathcal{O}_{j}$ associated to plane algebroid curves defined over an infinite field $K$. For every $j=1, \ldots, c$, let $S_j \subseteq \N^{d_j}$ denote the value semigroup of $\mathcal{O}_{j}$. For every $j$, $S_j$ is a local good semigroup (or a numerical semigroup).
The value semigroup of $R$ is $S= S_1 \times \cdots \times S_c \subseteq \N^{d}$ where $d = d_1 + \cdots + d_c$.

Let $\om=(\omega_1, \ldots, \omega_d) $ be an element of $S$  such that $\omega_i > 0$ for every $i=1, \ldots, d$. 
%such that $\omega_k \neq 0$ for every $k$ \bc do we need all components nonzero? \ec. 
Let $A$ be the Ap\'{e}ry set of $S$ with respect to $\om$ and set $N:= \omega_1 + \cdots + \omega_d$.
The set $A$ can be partitioned as $\bigcup_{i=0}^{N-1}A_i$.
Let $F \in R$ be an element of value $\om$.

 \begin{lemma}
    \label{lemmanodi}
    Let $h_1, \ldots, h_t \in R$ with $t \leq N-1$ be such that for every $j$,
    \begin{itemize}
    \item $v(h_j)= \al_j \in A$,
    \item $\al_j < \al_{j+1}$,
    \item if $\al_k \in \Delta^S_U(\al_j)$ for some $k > j$ and $U \neq \emptyset$, then $ \widetilde{\Delta}^S_{\widehat{U}}(\al_j) \subseteq A $.
    \end{itemize}
        %$v(h_j)= \al_j \in A$ for every $j$ and, if $\al_k \in \Delta^S_U(\al_j)$ for some $k > j$ and $U \neq \emptyset$, then $ \widetilde{\Delta}^S_{\widehat{U}}(\al_j) \subseteq A $. 
    %\bc in particular if $\Delta^S(\al_j) \subseteq A$ (i.e. they are "nodes") \ec
     Then the images of $h_1, \ldots, h_t$ modulo $(F)$ are linearly independent over $K \cong \frac{K[[F]]}{(F)}$.
   \end{lemma}
\begin{proof}
Call $\overline{h_j}$ the image of $h_j$ modulo $(F)$.
Suppose $ \sum^{t}_{j=1}a_j\overline{h_j} = 0$ for some $a_k \in K$ not all equal to zero. Then $H:= \sum^{t}_{j=1}a_jh_j \in (F)R$ and therefore $v(H) \not \in A$. 
%and  $v(H) > v(a_jh_j) = \al_j$ for every $j$.  
It follows that at least two coefficients $a_j$ are nonzero and without loss of generality we can assume $a_1, a_2 \neq 0$. Clearly $\al_1 \not \ll \al_2$, otherwise we would have $v(H)= v(a_1h_1)= \al_1 \in A$. Thus $\al_2 \in \Delta^S_U(\al_1)$ for some $U \neq \emptyset$. Since $\al_2 \leq \al_j$ for $j > 2$, it follows that $v(H) \in \widetilde{\Delta}^S_{\widehat{U}}(\al_1) \subseteq A.$ This is a contradiction.
\end{proof}

\begin{setting}
\label{setting}
 \rm  Let $R$, $A$ and $F$ be defined as above. For an element $G \in R$ not divisible by $F$, set $R_0= K$ and for $i = 1, \ldots, N-1$, 
\begin{equation}
R_i = K[[F]]+K[[F]]G + \cdots + K[[F]]G^{i}.
\end{equation}
Similarly set $T_0 = K$ and for $i = 1, \ldots, N-1$,
%\begin{equation}
%T_i =  \left\lbrace G^{i} + \phi \mbox{ such that } \phi \in K[[F]] + \ldots + G^{i-1}K[[F]] \mbox{ and } v(G^i + \phi) \not \in  v(T_{i-1}) \ec  \right\rbrace,
%\end{equation}
\begin{equation}
T_i =  \left\lbrace G^{i} + \phi \, | \, \phi \in R_{i-1} \mbox{ and } v(G^i + \phi) \not \in  v(R_{i-1}) \ec  \right\rbrace,
\end{equation}
We want to prove that we can find $G$ in such a way that $R= R_{N-1}$ and the equality $v(T_i)=A_i$ holds for every $i$.  More precisely, we will prove the two following theorems:

\begin{theorem} 
\label{semilocalring1part}
%Let $\mathcal{O}$, $A$ and $F$ be defined as above. 
Adopt the notation of Setting \ref{setting}. %and let $G$ be defined as in Proposition \ref{propertyofG}.
Then there exists $G \in R$ such that
\begin{equation}
R = K[[F]]+K[[F]]G + \cdots + K[[F]]G^{N-1}.
\end{equation}
\end{theorem}

\begin{theorem}
    \label{semilocalring2part}
    Adopt the notation of Setting \ref{setting} and define $G$ as in Theorem \ref{semilocalring1part}.
%Moreover, setting $R_0= T_0= K$, and inductively for $i = 1, \ldots, N-1$, 
%\begin{equation}
%R_i = K[[F]]+GK[[F]] + \cdots + G^{i}K[[F]],
%\end{equation}
%\begin{equation}
%T_i =  \left\lbrace G^{i} + \phi \mbox{ such that } \phi \in K[[F]] + \ldots + G^{i-1}K[[F]] \mbox{ and } v(G^i + \phi) \not \in  v(T_{i-1}) \ec  \right\rbrace,
%\end{equation}
%\begin{equation}
%T_i =  \left\lbrace G^{i} + \phi \, | \, \phi \in R_{i-1} \mbox{ and } v(G^i + \phi) \not \in  v(R_{i-1}) \ec  \right\rbrace,
%\end{equation}
Then for every $i=0, \ldots, N-1$, $$A_i = v(T_i).$$ 
\end{theorem}

\begin{oss}
    \label{remarklocalcase}
   In the case where $R = \mathcal{O}_1$ is local these results follow by  Proposition \ref{discussion}, Remark \ref{discussion2} and Proposition \ref{3.8apery}. 
 \end{oss}
 
By the above Remark \ref{remarklocalcase}, the results of the two theorems hold in particular in the case $d=1$. Hence, to prove the Theorems \ref{semilocalring1part} and \ref{semilocalring2part}, we can work by induction on $d$, assuming that $R$ is not local. It is sufficient then, slightly changing the notation, to assume that $R \cong \mathcal{O}_{1} \times  \mathcal{O}_2$ with $\mathcal{O}_{1}$ not necessarily local and $\mathcal{O}_{2}$ local. The value semigroup of $R$ will be denoted by $S= S_1 \times S_2$ with $S_i \subseteq \N^{d_i}$ and $d=d_1+d_2$.

We can thus write $F=(F_1, F_2)$ and $\om= (\om^{(1)}, \om^{(2)})$. Also $A^{(i)}$ will denote the Ap\'{e}ry set of $S_i$ with respect to $\om^{(i)}\in S_i$ (the projection of $\om$ with respect to the coordinates in $S_i$).  The number of levels of $A^{(i)}$ is equal to $N_i$, where $N_i$ is the sum of the coordinates of $\om^{(i)}$.

For $h = (h_1, h_2) \in R$, we let $v(h)= (v^{(1)}(h_1), v^{(2)}(h_2)) $ denote the value of $h$ in the semigroup $S$.
\end{setting}

The next proposition explains how to construct the power series $G$ in the ring $R$.

\begin{prop}
\label{propertyofG}
 Adopt the notation of Setting \ref{setting}. Then there exists $G \in R$ such that, for every $j=0, \ldots, N-1$ and $\al \in A_j$, we can find $\phi \in R_{j-1}$ such that $v(G^j+\phi)= \al$.
\end{prop}

\begin{proof}
We divide the proof in three parts. First we prove the result for elements of the
form $\al = (\al^{(1)}, \boldsymbol{0})$ with $\al^{(1)} \in A^{(1)}$, then we
consider elements of the form $\al = (\al^{(1)}, \boldsymbol{0})$ with $\al^{(1)}
\not \in A^{(1)}$, and by analogy
we obtain the same results also for all the elements of the form $\al =
(\boldsymbol{0}, \al^{(2)})$ with $\al^{(2)} \in S_2$  (our proof is independent
of whether $S_i$ is local or not). Finally, we will deal with the case $\al =
(\al^{(1)}, \al^{(2)})$ with $\al^{(1)}, \al^{(2)} \neq \boldsymbol{0}$.

\medskip
As mentioned in the above paragraph, by induction on $d$, we can assume that
Theorems \ref{semilocalring1part} and \ref{semilocalring2part} hold for $S_1$ and
$S_2$ with respect to the elements $F_1$ and $F_2$. Hence, for $i=1,2$, there
exists $G_i \in \mathcal{O}_{i}$ such that
$$ \mathcal{O}_i = K[[F_i]]+ K[[F_i]]G_i + \cdots +K[[F_i]] G_i^{N_i-1}. $$

Before to treat each one of the described cases we prove the next statement:

\medskip
\begin{lemma}\label{newlemma1}
%{\bf Claim 1:}
Let $L$ be a finite set of elements of the form
$\al = (\al^{(1)},\boldsymbol{0})$, $\al^{(1)}\in A^{(1)}_j$, $j\le N_1-1$. Then,
for all but finitely many choices of $\beta\in K$ we have
$$
v(G^j+\phi(F,G))=\al
$$
for some $\phi\in R_{j-1}$ and $G=(G_1,\beta+G_2)$.
\end{lemma}

\begin{proof}[Proof of the Lemma.]
Let $\al= (\al^{(1)}, \boldsymbol{0})$ with $\al^{(1)} \in A^{(1)}_j$.
Using the fact that both Theorems \ref{semilocalring1part} and
\ref{semilocalring2part} hold for $S_1$, we can find $\phi(F_1, G_1) \in
\mathcal{O}_1$  of degree at most $j-1$ in $G_1$ such that $\al^{(1)} =
v^{(1)}(G_1^j + \phi(F_1, G_1)) $. Clearly, since $K$ is infinite, for all but
finitely many elements $\beta \in K$, the value $v^{(2)}$ of $(\beta + G_2)^j +
\phi(F_2, \beta + G_2)$ is equal to the zero element of $S_2$. For all these
choices of $\beta$ we have $\al = v(G^j+ \phi(F,G))$.
 Hence, fixing any finite set $L$, consisting of elements of the form
$(\al^{(1)}, \boldsymbol{0})$ with $\al^{(1)} \in A^{(1)}$, we can choose the
element $\beta \in K$ in such a way that all the elements in $L$ satisfy the
thesis of this lemma.
\end{proof}

Modifying $G$ as $(\beta+G_1, G_2)$ we can clearly obtain the analogous result, for
infinitely many choices of the same  $\beta$, for a finite set $L'$ consisting of
elements of the form $(\boldsymbol{0}, \al^{(2)})$ with $\al^{(2)} \in A^{(2)}$.

Let us now prove the proposition, considering the different described  cases for
$\al\in S$.

\medskip

\bf Case 1: \rm $\al= (\al^{(1)}, \boldsymbol{0})$ with $\al^{(1)} \in A^{(1)}_j$
(or analogously $\al= (\boldsymbol{0}, \al^{(2)})$ with $\al^{(2)} \in A^{(2)}_j$). \\
For $j=0$ the result is clear since we must have $\al=\boldsymbol{0}= v(1)$.
By induction  we can assume that $\al^{(1)} \in A^{(1)}_j$ for $j >0 $ and the
thesis holds for any $\be^{(1)}\in A^{(1)}_k$ with $k < j$.

Choosing the element $\he$ for the semigroup $S_1$ according to Lemma
\ref{congruencelemma}, by the above Lemma \ref{newlemma1}, we can assume also that
the thesis holds
for all the elements $(\al^{(1)}, \boldsymbol{0})$ with $\al^{(1)} \in A^{(1)}$ and
$(\al^{(1)}, \boldsymbol{0}) \leq (\he,\boldsymbol{0})$ (these elements form
obviously a finite set).

Thus we can assume that the element $\al$ is
such that $\alpha^{(1)}_i> \eta_i$ for some $i$.
Let $\te = \al^{(1)}\wedge \he$, $U =\{i : \alpha^{(1)}_i \ge \eta_i\}$, and
$V = \{i : \alpha^{(1)}_i< \eta_i\}= I_1\setminus U$.
Then, one has (see Proposition \ref{newprop})
$\al^{(1)} \in \widetilde{\Delta}^{\N^{d_1}}_V(\te) \subseteq A^{(1)}_j$ and also
$\widetilde{\Delta}^{\N^{d_1}}_V(\al^{(1)}) \subseteq
\widetilde{\Delta}^{\N^{d_1}}_V(\te) \subseteq A^{(1)}_j$.  Note that
every element of $\widetilde{\Delta}^{\N^{d_1}}_V(\te) \cup \lbrace \te \rbrace$ (in
particular $\al^{(1)}$) satisfies the assumptions of Lemma \ref{congruencelemma}
choosing any index $i \in U$.

%In this case $\al^{(1)}$ is an
%element of an infinite
%subspace in $A^{(1)}_j$ (subspaces are defined Defintion \ref{subspacedef}), and we
%can write
%$\al^{(1)} \in \widetilde{\Delta}^{\N^{d_1}}_V(\te) \subseteq A^{(1)}_j$ where
%$\te$ is  such that $\he \in \Delta^{E_1}_U(\te)$ for $U:= I_1 \setminus V$.
%Clearly we get $\alpha^{(1)}_i \geq \eta_i$ if and only if $i \in U$. In particular
%every element of $\widetilde{\Delta}^{\N^{d_1}}_V(\te) \cup \lbrace \te \rbrace$
%satisfies the assumptions of Lemma \ref{congruencelemma} choosing any index $i \in U$.

Now, let us prove the next:

\begin{lemma}\label{newlemma2}
%{\bf Claim 2:}
Let $\eps \in \widetilde{\Delta}^{\N^{d_1}}_{V}(\te)\subset A^{(1)}_j$. Then
there exists $(\de, m \om^{(2)})\in \Delta^E_{U}(\eps,\boldsymbol{0})$ with $m\ge 1$, such
that
$(\de, m \om^{(2)}) = v(\psi)$ for some $\psi\in R_{j-1}$.
\end{lemma}

%There exists $\be \geq \al^{(1)}$ such that $\be \in \widetilde{\Delta}^{\N^{d_1}}_V(\te)$
%and $(\be, \boldsymbol{0})= v(G^j+ \phi_{\be})$ for some  $\phi_{\be} \in R_{j-1}$.

\begin{proof}[Proof of the Lemma.]
Let $i \in U$. By Lemma \ref{congruencelemma} there exists
an element $\de^{(i)} \in \Delta^{E_1}_i(\eps)$ such that $\de^{(i)} = m_i
\om^{(1)} + \be^{(i)}$, with $m_i \geq 1$ and $\be^{(i)} \in A^{(1)}_{k_i}$ with
$k_i < j$. By the inductive hypothesis on $j$ we know that $(\be^{(i)},
\boldsymbol{0})= v(\Phi_i)$ with $\Phi_i \in R_{j-1}$.
%$\Phi_i \in T_{k_i} \subseteq R_{j-1}$.
Since $\om  \gg \boldsymbol{0}$  we get
$$ m_i \om + (\be^{(i)}, \boldsymbol{0}) = (\de^{(i)}, m_i \om^{(2)}) =
v(F^{m_i}\Phi_i) \in v(R_{j-1}) \cap \Delta^E_i((\eps, \boldsymbol{0})). $$
Setting $m= \min_{i \in U} \{m_i\}$, we consider the infimum
$$
\bigwedge_{i \in U} (\de^{(i)}, m_i \om^{(2)}) = (\de, m \om^{(2)} ) \in
\Delta^E_U((\eps, \boldsymbol{0}) )\; .
$$
For some choice of elements $z_i \in K$, we know that $(\de, m \om^{(2)} ) =
v(\sum_{i \in U} z_i F^{m_i}\Phi_i)$. Set $\psi:= \sum_{i \in U} z_i F^{m_i}\Phi_i
\in R_{j-1}$.
Note that, if $j\in V$ then $\delta^{(i)}_j > \epsilon_j$ for all $i\in U$ and therefore
$\delta_j>\epsilon_j$; on the other hand, if $j\in U$ then $\delta_j=\epsilon_j$.
\end{proof}

Now, let us apply the above Lemma \ref{newlemma2} to the element $\eps=\te$.
Since $\te \leq \he$ we know that $(\te , \boldsymbol{0}) = v(G^j + \phi_{\te})$ for some
$\phi_{\te} \in R_{j-1}$.
Let us fix an index  $k \in U$. Since $\delta_k=\theta_k$
we can choose $t_k \in K$ such that
$v(G^j + \phi_{\te}+ t_k\psi )= (\te', \boldsymbol{0}) > (\te, \boldsymbol{0})$
with $\theta'_k > \theta_k$. Note that, if $j\in V$ then
$\theta'_j =\theta_j$, hence $\te'\in \widetilde{\Delta}^{\N^{d_1}}_V(\te)$.

Iterating this process, replacing each time $\te$ by $\te'$ and possibly using the other
indices  $k\in U$, we can find an element $\te'\in \widetilde{\Delta}^{\N^{d_1}}_V(\te)$ with arbitrarily large coordinates with respect to the
indices in $U$ such that %$\te' \in {\Delta}^{\N^{d_1}}_V(\te)$ and 
$(\te', \boldsymbol{0})= v(G^j+ \phi_{\be})$ for some  $\phi_{\be} \in R_{j-1}$.

Going back to the element $\al^{(1)} \in \widetilde{\Delta}^{\N^{d_1}}_V(\te)$,
in particular we can
find $\be \geq \al^{(1)}$ such that $\be \in \widetilde{\Delta}^{\N^{d_1}}_V(\te)$
and $(\be, \boldsymbol{0})= v(G^j+ \phi_{\be})$ with $\phi_{\be} \in R_{j-1}$.
Explicitly, we can say that $\be \in \Delta^{S_1}_W(\al^{(1)})$ with $W \supseteq V$.

\medskip

Furthermore, by the Lemma \ref{newlemma2} applied to the element
$\eps=\al^{(1)}$ we can construct an element
$(\de', \boldsymbol{\tau}) \in \Delta^E_{U}((\al^{(1)}, \boldsymbol{0}))$  such that
$(\de', \boldsymbol{\tau}) = v(\psi')$ with $\psi' \in R_{j-1}$ (and
$\boldsymbol{\tau} \gg \boldsymbol{0}$). It is easy to observe that $ (\al^{(1)},
\boldsymbol{0}) =  (\be, \boldsymbol{0}) \wedge  (\de', \boldsymbol{\tau})$.
Thus we can choose $z \in K$
such that $v(G^j+ \phi_{\be} + z\psi')= (\al^{(1)}, \boldsymbol{0})$. This shows
that $(\al^{(1)}, \boldsymbol{0})$ is the value of some element of the form $G^j+
\phi$ with $\phi \in R_{j-1}$ and completes the proof of the Case 1.
%Thus $\de:= \wedge_{i \in U} \de^{(i)} \in \Delta^{E_1}_U(\te)$.

\medskip

 \bf Case 2: \rm $(\al^{(1)}, \boldsymbol{0}) \in A_j$ with $\al^{(1)} \not \in
A^{(1)}$ (or analogously $\al= (\boldsymbol{0}, \al^{(2)}) \in A_j$ with $\al^{(2)} \not \in A^{(2)}_j$). \\
Suppose $\al^{(1)}$ to be nonzero.
By Theorem \ref{thmnonlocallevels} also in this case we have $\lambda(\al^{(1)})=
j > 0$. By definition of $\lambda$ we can write $\al^{(1)}) = m \om^{(1)}+\te $ for $m \geq 1 $ and $\te \in A^{(1)}_k$ with $k < j$.
 If $j < N_1$, then by Lemma \ref{lambda}, there exists $\de \in A^{(1)}_j$ such
that $\al^{(1)} < \de$.
%Furthermore there exists $m \geq 1 $ and $\te \in A^{(1)}_k$ with $k < j$ such that $\al = m \om^{(1)}+\te $.
As a consequence of Case 1, we know that $(\de, \boldsymbol{0}) = v(G^j
+\phi)$ with $\phi \in R_{j-1}$ and $(\te, \boldsymbol{0}) = v(G^k +\psi)$ with
$\psi \in R_{k-1}$.
Since $K$ is infinite we can find a nonzero constant $z \in K$ such that
$(\al^{(1)}), \boldsymbol{0})= v(G^j + \phi + zF^m (G^k+ \psi))$. The result now
follows since by construction $\phi + zF^m (G^k+ \psi) \in R_{j-1}$.
 %Thus $\alpha_1 = m \om^{(1)} + w^{(1)}_k$ with $k < j$, $m \geq 1$ and $v(G^j + \Phi_j + F^m (G^k+ \Phi_k)) = (w^{(1)}_j, 0) \wedge (\alpha_1, m\om^{(2)}) = (\alpha_1, 0). $ By construction $ \Phi_j + F^m (G^k+ \Phi_k) \in R_{j-1}.$
 If instead $j=N_1$, we use the fact that we can express $G_1^{N_1}=
\sum_{i=0}^{N_1-1}h_{i}(F_1)G_1^i$ and the choice of the element $\beta \in K$
%(defined in Theorem \ref{semilocalring1part})
 can be made in such a way that $$v^{(2)}\left( (\beta+G_2)^{N_1}-
\sum_{i=0}^{N_1-1}h_{i}(F_2)(\beta+G_2)^i \right)=0.$$ Thus %again, after
writing again $(\te, \boldsymbol{0}) = v(G^k +\psi)$ with $\psi \in R_{k-1}$,
we obtain
$$\al= (\al^{(1)}, \boldsymbol{0})= v\left(G^{N_1}- \sum_{i=0}^{N_1-1}h_{i}(F)G^i +
F^m(G^k+ \psi)\right).$$ As before $ \sum_{i=0}^{N_1-1}h_{i}(F)G^i + F^m(G^k+ \psi)
\in R_{j-1}.$

%Hence, we reduced to prove the same result only in the case $\al^{(1)} \in A^{(1)}.$
%\bc TO DO \ec \\

\medskip
In both Cases 1 and 2, we get the same results for the elements of the form $\al =
(\boldsymbol{0}, \al^{(2)})$. Indeed, we can proceed in the same way working over
the components corresponding to $S_2$ and replacing $G$ by $G - (\beta,\beta)$ in
all the formulas (again the choice of $\beta$ at the beginning of the proof can be
made generic enough to satisfy all the needed conditions). We finally consider the
general case: 

\medskip
\bf Case 3: \rm $\al=(\al^{(1)}, \al^{(2)}) \in A_j$ with $\al^{(1)}, \al^{(2)}
\neq 0$.  \\
%Consider now the general case $\al=(\al^{(1)}, \al^{(2)}) \in A_j$ with $\al^{(1)}, \al^{(2)} \neq 0$.
We can say that $\lambda(\al^{(1)}, \boldsymbol{0})=k$, $\lambda(\boldsymbol{0},
\al^{(2)})=l$ with  $k,l \geq 1$. By Theorem \ref{thmnonlocallevels}, $k+l=j$.
By what proved in the previous cases $(\al^{(1)}, \boldsymbol{0}) = v(G^{k} +
\Phi)$ and $(\boldsymbol{0}, \al^{(2)}) = v(G^l + \Psi)$ for opportune choices of
$\Phi \in R_{k-1}$ and $\Psi \in R_{l-1}$. It follows that $\al = v((G^{k} + \Phi)
(G^{l} + \Psi))= v(G^j + \xi)$ with $\xi \in R_{j-1}$.

This concludes the proof of the proposition.
\end{proof}

We prove now Theorem \ref{semilocalring1part}.

%\begin{theorem} 
%\label{semilocalring1part}
%Let $\mathcal{O}$, $A$ and $F$ be defined as above. 
%Take the notation of Setting \ref{setting} and let $G$ be defined as in Proposition \ref{propertyofG}.
%Then
%\begin{equation}
%\mathcal{O} = K[[F]]+GK[[F]] + \cdots + G^{N-1}K[[F]].
%\end{equation}
%\end{theorem}

%\subsection*{segno}

%\medskip

\begin{proof}(proof of Theorem \ref{semilocalring1part}) \\
We know that $R$ is a $K[[F]]-$module and, since the quotient ring $ \frac{R}{(F)R} $ is a $K$-vector space of dimension $N_1+N_2$, the ring $R$ is minimally generated as module over $K[[F]]$ by $N=N_1+N_2$ elements. For $H \in R$, denote by $\overline{H}$ the image of $H$ in the quotient $ \frac{R}{(F)R} $.

Let $G$ be defined as in Proposition \ref{propertyofG}.
To prove the theorem we need to show that %find $G \in R$ such that 
$$ \overline{1}, \overline{G},\overline{G}^2, \ldots, \overline{G}^{N_1+N_2-1} $$ are linearly independent over $K$.
We use now Remark \ref{chainofnodes} to construct a sequence of elements of  $S$ $$ \boldsymbol{0}= \al^{(0)} < \al^{(1)} < \cdots < \al^{(j)} < \cdots < \al^{(N-2)} < \al^{(N-1)}, $$ such that %for every $i=0, \ldots, N-1$, 
$\al^{(i)} \in A_i$ and, for every $k < j$, if $\al^{(j)} \in \Delta^S_U(\al^{(k)})$ for some $U \neq \emptyset$, then $\widetilde{\Delta}^S_{\widehat{U}}(\al^{(k)}) \subseteq A$. By Proposition \ref{propertyofG}, $\al^{(j)}$ is the value of an element of the form $h_j:= G^j +\phi$ with $\phi \in R_{j-1}$. The elements $h_0, \ldots, h_{N-1}$ satisfy the hypothesis of Lemma \ref{lemmanodi}. Thus their images modulo $(F)$ are linearly independent over $K$. By definition of the subsets $R_j$, it follows that also $ \overline{1}, \overline{G},\overline{G}^2, \ldots, \overline{G}^{N_1+N_2-1} $ are linearly independent over $K$. This proves the theorem. 
\end{proof}

%\begin{theorem}
 %   \label{semilocalring2part}
 %   Take the notation of Setting \ref{setting} and let $G$ be defined as in Theorem \ref{semilocalring1part}.
%Moreover, setting $R_0= T_0= K$, and inductively for $i = 1, \ldots, N-1$, 
%\begin{equation}
%R_i = K[[F]]+GK[[F]] + \cdots + G^{i}K[[F]],
%\end{equation}
%\begin{equation}
%T_i =  \left\lbrace G^{i} + \phi \mbox{ such that } \phi \in K[[F]] + \ldots + G^{i-1}K[[F]] \mbox{ and } v(G^i + \phi) \not \in  v(T_{i-1}) \ec  \right\rbrace,
%\end{equation}
%\begin{equation}
%T_i =  \left\lbrace G^{i} + \phi \, | \, \phi \in R_{i-1} \mbox{ and } v(G^i + \phi) \not \in  v(R_{i-1}) \ec  \right\rbrace,
%\end{equation}
%Then for every $i=0, \ldots, N-1$, $$A_i = v(T_i).$$ 
%\end{theorem}

Before proving Theorem \ref{semilocalring2part}, we need to prove several lemmas.

\begin{lemma}
    \label{lemmaTi}
    Take the notation of Setting \ref{setting}.
    Let $\al, \be \in v(T_i)$ for some $i=0, \ldots, N-1$.  If $\al\neq \be$, then $\al$ and $\be$ are incomparable with respect to the partial order relation $\le\le$.
\end{lemma}

\begin{proof}
    Write $\al= v(G^i+\phi)$ and $\be=v(G^i+\psi)$ for $\phi, \psi \in R_{i-1}$. If by way of contradiction $\al \ll \be$, we would have $\al = v(G^i+\phi-G^i-\psi) = v(\phi-\psi) $ and this would contradict the definition of $T_i$.
\end{proof}

%\begin{lemma}
 %   \label{level-uOLD}
 % Take the notation of Setting \ref{setting} and let $G$ be defined as in the proof of Theorem \ref{semilocalring1part}. For $i \in \lbrace 1,2 \rbrace$, let $\phi= \sum_{k=0}^j a_k(F_i)G_i^k \in \mathcal{O}_i $ be a power series not divisible by $F_i$. %Let $\al= v(\phi)$.
  %Suppose that $\al= v(\phi) \in A^{(i)}$. 
  %Then $\lambda(v(\phi)) \leq j$. \bc this should be proved in general for arbitrary good semigroups \ec
%\end{lemma}

%\begin{proof}
 %   Since $\phi$ is not divisible by $F_i$, at least one of the series $a_k(F_i)$ has nonzero constant term. Using the assumption that the thesis of Theorems \ref{semilocalring1part} and \ref{semilocalring2part} hold for the ring $\mathcal{O}_i$, we can use Weierstrass' Preparation Theorem (\bc needs $\mathcal{O}_i$ local \ec) and reduce to the case where 
     %$\phi = G^h + \psi$ with $h \leq j$ and $\psi$ has degree in $G_i$ at most $h-1$. Defining the %sets $T_0, \ldots, T_{N_i-1}$ also for the ring $\mathcal{O}_i$, we can find $\xi$ of degree in $G_i$ at most $h-1$ such that $G^h + \xi \in T_h$. Since by definition of $T_h$, $v(G^h + \xi) \neq v(\psi - \xi)$, we obtain $v(\phi) \leq v(G^h + \xi)$ (\bc this uses $S_i$ numerical\ec). By assumption on $\mathcal{O}_i$, $v(G^h + \xi) \in A^{(i)}_h$. Lemma \ref{lambda} implies that $\lambda(\al) \leq j$.
%\end{proof}

\begin{lemma}
    \label{level-u}
    Let $R$ be the local ring of a plane curve and let $S$ be its value semigroup. Let the elements $F,G$ and the subsets $R_i, T_i$ be defined as in Setting \ref{setting} and Remark \ref{remarklocalcase}. For $j \leq N-1$, let $\phi= \sum_{k=0}^j a_k(F)G^k \in R$ be a power series not divisible by $F$. %Let $\al= v(\phi)$.
  %Suppose that $\al= v(\phi) \in A^{(i)}$. 
  Then $\lambda(v(\phi)) \leq j$. 
\end{lemma}

\begin{proof}
In the case $j=0$, $\phi$ is a power series in $K[[F]]$ not divisible by $F$ and $v(\phi)= \boldsymbol{0}$. It follows that $\lambda(v(\phi)) =0$. Thus, we can argue by induction and assume the thesis true for all the power series having degree in $G$ strictly smaller than $ j$.
    Since $\phi$ is not divisible by $F$, at least one of the series $a_k(F)$ has nonzero constant term. Thanks to the fact that the ring $R$ is local, we can use Weierstrass' Preparation Theorem to write $\phi = u(F,G) (G^h + \psi)$ with $h \leq j$, $\psi \in R_{j-1}$ and $v(u(F,G))=\boldsymbol{0}$. If $h < j$ we can conclude by inductive hypothesis.
    From this we can reduce to the case where 
     $\phi = G^j + \psi$ with $\psi \in R_{j-1}$. %Defining the sets $T_0, \ldots, T_{N_i-1}$ also for the ring $R_i$,
     Now set $\al= v(G^j+\psi) $. By way of contradiction suppose $\lambda(\al) > j$. Hence, by definition of $\lambda$ and by Lemma \ref{lemmanodominore} %\cite[Proposition 2.10]{GMM2} 
     we can find $\be \in A_j$ such that $\al > \be$. Since $R$ is local, by Remark \ref{remarklocalcase} we know that Theorem \ref{semilocalring2part} holds for $R$ and we get $A_j=v(T_j)$.
     Hence, we can find $\xi \in R_{j-1}$ such that $G^j + \xi \in T_j$ and $v(G^j + \xi)= \be$. Set $\de= v(\psi - \xi)$ and observe that $\al = v(G^j + \xi + \psi -\xi)$.
     By definition of $T_j$, $\be \neq \de$. For any component $i$ such that $\delta_i \neq \beta_i$, we get $ \min(\delta_i, \beta_i) = \alpha_i \geq \beta_i $ and thus $\alpha_i = \beta_i$. This implies that $\be < \de$. Now if $\be \ll \de $ we get the contradiction $\al = \be$. Therefore there exists a non-empty set of indices $U$ such that $\al \in \Delta^S_U(\be)$ and $\de \in \widetilde{\Delta}^S_{\widehat{U}}(\be)$. Now if $\al \in E$, clearly $\widetilde{\Delta}^S_{\widehat{U}}(\be) \subseteq A$ by property (G1), since $\de \wedge \al = \be \not \in E$. If $\al \in A$, then $\al \in A_l$ with $l > j $ and we can use Lemma \ref{lemmanodominore} to choose $\be$ in such a way that $\widetilde{\Delta}^S_{\widehat{U}}(\be) \subseteq A$. In any case $\de \in v(R_{j-1}) \cap A$ and therefore $F$ does not divide $\psi - \xi$. By inductive hypothesis $\lambda(\de) \leq j-1$ implying that $\de \in A_k$ with $k < j$. This is a contradiction since $\de > \be$.
     %hence $v(\phi) \leq v(G^j + \xi)$ (\bc this uses $S_i$ numerical\ec). 
      %By assumption on $\mathcal{O}_i$, $v(G^h + \xi) \in A^{(i)}_h$. Lemma \ref{lambda} implies that $\lambda(\al) \leq j$.
\end{proof}

\begin{lemma}
    \label{levelofproduct-u*f}
  Adopt the notation of Setting \ref{setting} and let $G$ be defined as in the proof of Theorem \ref{semilocalring1part}. Let $G^j + \phi \in T_j$. Suppose that $v(G^j+ \phi) \in A_j$ and there exists $u= u(F,G) \in R$ of degree $k$ in $G$ such that $v(u)=(\boldsymbol{0},\te)$ and $ v(uG^j+u\phi) \in A_h$. Then $h \leq j+k$.
\end{lemma}

\begin{proof}   
    Let $Y_1$ and $Y_2$ be the components of $uG^j+u\phi$ with respect to the direct product $\mathcal{O}_1 \times \mathcal{O}_2$. We recall that, from what written right after Remark \ref{remarklocalcase}, we can assume $\mathcal{O}_2$ to be local. By Theorem \ref{thmnonlocallevels}, $h= \lambda(v^{(1)}(Y_1))+ \lambda(v^{(2)}(Y_2)). $ Similarly write $j=j_1 + j_2$ where $j_1$ and $j_2$ are the values of the function $\lambda$ applied to the two components of $G^j+ \phi$.  Since the first component of $u$ has value zero we get $\lambda(v^{(1)}(Y_1)) = j_1$.  We need to prove that $\lambda(v^{(2)}(Y_2)) \leq k + j_2$. Applying Lemma \ref{level-u} to the second component of $u$ in the local ring $\mathcal{O}_2$ we get $\lambda(\te) \leq k$. 
    %$\lambda(a) = 0 + \lambda(a) = \lambda((0,a)) \leq k$.
    Thus, it is sufficient to prove that, if $\al, \be \in S_2$, $\lambda(\al) = i$ and $\lambda(\be) = k$, then $\lambda(\al + \be) \leq i+k$. Since the maximal value of $\lambda(\de)$ for $\de \in S_2$ is $N_2$, we can reduce to assume $i+k < N_2$.
    By Lemma \ref{lambda} we can replace $\al$, $\be$ by $\al' \in A_i$ and $\be' \in A_k$ such that $\al \leq \al'$ and $\be \leq \be'$ (in particular $\lambda(\al + \be) \leq \lambda(\al' + \be')$). Hence, let us assume that $\al \in A_i$ and $\be \in A_k$. By assumption on $R_2$, %defining the sets $T_0, \ldots, T_{N_2-1}$ for this ring, 
    we can find $G_2^i+ \psi_{i-1}$ and $G_2^k+ \psi_{k-1}$ having values respectively equal to $\al$ and $\be$. Then $\al + \be = v(G_2^{i+k} + \psi)$ for some $\psi$ having degree at most $i+k-1$ in $G_2$. To conclude we can now apply Lemma \ref{level-u} at the element $G_2^{i+k} + \psi \in R_2$.
    %Pick now $\te \in A_{i+k}$ and write $\te= v(G_2^{i+k} + \xi)$ for $G_2^{i+k} + \xi \in T_{i+k}$. Since $v(G_2^{i+k} + \xi) \neq $
    %$j_1= \lambda(v(G_1^j+ \phi(F_1,G_1)))$ and $j_2= \lambda(v((G_2+ \beta)^j+ \phi(F_2,G_2+\beta)))$
\end{proof}

%\begin{theorem}
 %   \label{semilocalring2part}
  %  Take the notation of Setting \ref{setting} and let $G$ be defined as in the proof of Theorem \ref{semilocalring1part}.
%Moreover, setting $R_0= T_0= K$, and inductively for $i = 1, \ldots, N-1$, 
%\begin{equation}
%R_i = K[[F]]+GK[[F]] + \cdots + G^{i}K[[F]],
%\end{equation}
%\begin{equation}
%T_i =  \left\lbrace G^{i} + \phi \mbox{ such that } \phi \in K[[F]] + \ldots + G^{i-1}K[[F]] \mbox{ and } v(G^i + \phi) \not \in  v(T_{i-1}) \ec  \right\rbrace,
%\end{equation}
%\begin{equation}
%T_i =  \left\lbrace G^{i} + \phi \, | \, \phi \in R_{i-1} \mbox{ and } v(G^i + \phi) \not \in  v(R_{i-1}) \ec  \right\rbrace,
%\end{equation}
%Then $A_i = v(T_i)$ for every $i=0, \ldots, N-1$.
%\end{theorem}

We are now ready to prove Theorem \ref{semilocalring2part}.

\begin{proof} (proof of Theorem \ref{semilocalring2part}).
    Starting from the fact that $A_0 = \lbrace (0,0) \rbrace = v(K)=  v(T_0)$, we prove that $A_j = v(T_j)$ for every $j=0, \ldots, N-1$ by induction. Fixed $j >0$, assume that $A_k = v(T_k)$ for all $k< j$. Thanks to Proposition \ref{propertyofG}, we know that for every $\al \in A_j$ there exists $\phi \in R_{j-1}$ such that $\al= v(G^j+\phi)$.
Thus, we only need to prove that, given $\phi \in R_{j-1}$, the following conditions are equivalent:
\begin{enumerate}
    \item[(i)] $v(G^j + \phi) \in A_j$.
    \item[(ii)] $G^j+ \phi \in T_j$.
\end{enumerate}
%n $v(G^j + \phi) \in A_j$ if and only if $G^j+ \phi \in T_j$.
%the following facts:
%\begin{itemize}
   % \item[(i)] For every $\al \in A_j$, we can find $\phi \in R_{j-1}$ such that $v(G^j+\phi)= \al$.
 %   \item[(i)] Given $\phi \in R_{j-1}$, if $v(G^j + \phi) \in A_j$, then $G^j+ \phi \in T_j$.
 %   \item[(ii)] Given $\phi \in R_{j-1}$, if $G^j + \phi \in T_j$, then $v(G^j + \phi) \in A_j$. 
%\end{itemize}
%Let us now prove (ii). 
Let us prove (i) $\Longrightarrow$ (ii). 
%First suppose $v(G^j + \phi) \in A_j$.
Assume by way of contradiction $G^j+ \phi \not \in T_j$ and set $\al= v(G^j+\phi)$. Hence there exists $H \in R_{j-1}$ such that $v(H)=\al$.  Write $H=H(F,G)= \sum_{k=0}^{j-1} a_k(F) G^k$. Since $v(H) \in A$, $H$ is not divisible by $F$ and thus at least one of the power series $a_k(F)$ has nonzero constant term. We can apply Weierstrass' Preparation Theorem on the power series $\sum_{k=0}^{j-1} a_k(x) y^k$ in the local formal power series ring $K[[x,y]]$. This gives $H(x,y) = u(x,y)(\sum_{k=0}^{h-1} b_k(x) y^k + y^h)$ for $h \leq j-1$  and $u(x,y)$ with nonzero constant term. Mapping to the ring $R$, we obtain $H = u(F,G)(\sum_{k=0}^{h-1} b_k(F) G^k + G^h)$, where still $u:= u(F,G)$ has nonzero constant term but is not necessarily a unit. 
In particular, by definition of $F$ and $G$ we know that $v(u)=(0,a) $ for some $a \in S_2$. Set $G^h+ \psi= \sum_{k=0}^{h-1} b_k(F) G^k + G^h.$ Clearly, since $ (0,a) + v(G^h+ \psi) \in A$, then also $\be := v(G^h + \psi) \in A$.   Possibly iterating the same process finitely many times, replacing $G^j+ \phi $ by $G^h+\psi$, we can reduce to the case where $G^h + \psi \in T_h$ (eventually $R_0 = T_0$). By inductive hypothesis we get $\be \in A_h$. The division argument of Weierstrass' Preparation Theorem implies that $u=(u_1, u_2)$ is a polynomial in $G$ of degree $j-1-h$. By Lemma \ref{levelofproduct-u*f} we obtain $\al= v(H) = v(uG^h+u\psi) \in A_i$ with $i \leq (j-1-h)+h=j-1$. This contradicts the assumption of having $\al= v(G^j + \phi) \in A_j$. 

%Writing $\be=(b_1, b_2)$ and using Theorem \ref{thmnonlocallevels}, we know that $h= \lambda(b_1) + \lambda(b_2)$. Since the first component of $u$ has value zero we can also write $j= \lambda(\al)=  \lambda(b_1) + \lambda(v(L))$, where $L$ is the product of the second component of $u$ with the second component of $G^h+\psi$. Moreover, $u=(u_1, u_2)$ is a polynomial in $G$ of degree $j-1-h$ and by \bc Lemma ... \ec $\lambda(v(u)) = \lambda(v(u_2)) \leq j-1-h$. \rc we would need now $\lambda(\mbox{second component of } G^h + \psi)$ to be exactly equal to the degree of that component in $G_2$
% \ec \\

%We finally prove (iii).
\medskip 
We prove now (ii) $\Longrightarrow$ (i).
%Now suppose that $G^j + \phi \in T_j$.
Let $\al= v(G^j+\phi)$ and suppose first that $\al \not \in A$. Hence, we can write $\al= m \om + \de$ with $\de \in A_h$ and $m \geq 1$. If $h < j$, by inductive hypothesis, we can find $G^h+ \psi \in T_{h}$ such that $\de= v(G^h+ \psi)$. It follows that $\al = v(F^m(G^h+ \psi))$ and this contradicts the definition of $T_j$. If instead $h \geq j$, there exists $\be \in A_j$ such that $\be \leq \de $. Hence, $\al \gg \be$ and Proposition \ref{propertyofG} together with the implication (i) $\Longrightarrow$ (ii) allows us to find $G^j+ \psi \in T_{j}$ such that $\be= v(G^j+ \psi)$. \ec This yields a contradiction by Lemma \ref{lemmaTi}.

Suppose then $\al \in A_h$ for some $h$. By inductive hypothesis, since the sets $v(T_i)$ are disjoint by definition, we must have $h \geq j$. If $h > j$, by Lemma \ref{lemmanodominore} %by \cite[Proposition 2.10]{GMM2}
we can find $\be \in A_j$ such that $\be < \al$. As before we can find $G^j+ \psi \in T_{j}$ such that $\be= v(G^j+ \psi)$. If $\al \gg \be$, we conclude as previously using Lemma \ref{lemmaTi}. Otherwise we have $\al \in \Delta_F^S(\be)$ and we can use
Lemma \ref{lemmanodominore}
to assume also that $ \widetilde{\Delta}^S_{\widehat{F}}(\be) \subseteq A$. From this we get $\de := v(G^j+\phi - G^j - \psi) \in \widetilde{\Delta}^S_{\widehat{F}}(\be) \subseteq A.$ In particular $\de \in v(R_{j-1}) \cap A$. %Lemma \ref{valueofRi} implies that $\de \in A_k$ with $k < j$. 
To conclude we prove that $v(R_{j-1}) \cap A \subseteq \bigcup_{l=0}^{j-1} A_l$.
This will show that $\de \in A_l$ with $l < j$ in contradiction with the fact that $\de \geq \be$.
 For $\de \in v(R_{j-1}) \cap A$, arguing as in the proof of implication (i) $\Longrightarrow$ (ii), we write $\de = v(\sum_{k=0}^{i-1} a_k(F) G^k)$ and use Weierstrass' Preparation Theorem to get $\sum_{k=0}^{j-1} a_k(F) G^k = u(F,G) (G^s + \xi)$ such that $s < j$, $v(u(F,G))=(0,a)$ and $G^s+\xi \in T_s$. The same argument used previously shows that $\be \in  A_l$ with $l \leq j-1$.
%we can use (i)-(ii) to find $G^h+ \psi \in T_{h}$ such that $\al= v(G^h+ \psi)$. But this contradicts the definition of $T_h$ since now $v(G^h+ \psi)= v(G^j+\phi)$. Therefore we must have $h=j$. To conclude we suppose that $\al \not \in A$ and we prove that this yields a contradiction. 
%If $\al \not \in A$, we can write $\al= m \om + \de$ with $\de \in A_h$ and $m \geq 1$. If $h < j$ again we find $G^h+ \psi \in T_{h}$ such that $\de= v(G^h+ \psi)$. It follows that $\al = v(F^m(G^h+ \psi))$ and this contradicts the definition of $T_j$. If instead $h \geq j$, there exists $\be \in A_j$ such that $\be \leq \de $. Hence, $\al \gg \be$ and we can find $G^j+ \psi' \in T_{j}$ such that $\be= v(G^j+ \psi')$. This yields a contradiction by Lemma \ref{lemmaTi}. %Now we can write $G^j+\phi = (G^j+\psi')+(\phi - \psi')$. Since $\al \gg \be$ we must have $v(\phi-psi')=\be$. But this contradicts the fact that $G^j+ \psi' \in T_{j}$ and concludes the proof.
\end{proof}

%\bigskip

%\bc start a subsection now? \ec \\

\subsection{Ap\'ery's Theorem for semilocal blow-ups of plane algebroid curves}

Let $K$ be an infinite field and let $\mathcal V$ be a product of  local  rings of plane algebroid curves defined over $K$.
%the coordinate ring of a plane algebroid curves. 
Then it is well-known that:
\begin{itemize}
	\item $\mathcal V \cong \mathcal V_1\times\cdots\times \mathcal V_c\subseteq \overline{\mathcal V}\cong K[[t_1]]\times\cdots\times K[[t_d]]$ with $(\mathcal V_i,\textbf{m}_i)$ local rings, 
 %$K$ infinite 
 and $\overline{\mathcal V}$ a finite $\mathcal V$-module.
 \item $\mathcal V$ is reduced.
 \item $\mathcal V_i/\textbf{m}_i\cong K$.
\end{itemize}

We can always assume that $\mathcal V\cong \mathcal V_1\times \mathcal V_2$ with $\mathcal V_1$ not necessary local and $\mathcal V_2$ local. By Theorems \ref{semilocalring1part} and \ref{semilocalring2part}, we can write
$$
\mathcal V= K[[F]]+K[[F]]G + \cdots + K[[F]]G^{N-1}
$$
\noindent where $F$ is any element of $\mathcal V$ of value $\om=(\om^{(1)},\om^{(2)})$ with \ec  $\om^{(i)} \gg \boldsymbol 0$, and $G$ defined according to the proof of Proposition \ref{propertyofG}.

\begin{prop}
\label{blowup}
 The ring $\mathcal U= K[[F]]+K[[F]]H + \cdots + K[[F]]H^{N-1}$ with $H=G\cdot F$ is the local \ec   ring of a plane algebroid curve and its blowup, $\mathcal B(\mathcal U)$, is equal to $\mathcal V$.
 \end{prop}
\begin{proof} We first show that $\mathcal U$ is local with maximal ideal $(F,H)$. %Since 
The element $G^N\in \mathcal V$ satisfies a relation $G^N=a_0(F)+a_1(F)G+\cdots+a_{N-1}(F)G^{N-1}$. Hence 
$$
(*)\ \ \ H^N=a_0(F)F^N+a_1(F)F^{N-1}H+\cdots+a_{N-1}(F)FH^{N-1}.
$$
Let $\varphi:K[[x]][y]\longrightarrow \mathcal U$ be the surjective homomorphism defined by $\varphi(x)=F$ and $\varphi(y)=H$. Since $\mathcal V$ is minimally generated as $K[[F]]$-module by $\{1,G,\dots,G^{N-1}\}$, then necessarily $N$ is the minimal integer such that the powers $1, H, H^2, \ldots, H^N$ are linearly dependent over $K[[F]]$. 
%degree of $H$ in the equation (*) above \ec. 
Hence %$\ker\varphi=(f)$ where
$$
f=y^N-a_0(x)x^N-a_1(x)x^{N-1}y-\cdots-a_{N-1}(x)xy^{N-1}
$$ is an irreducible element of $K[[x]][y]$ and therefore
$\ker\varphi=(f)$, and
$\mathcal U\cong\frac{K[[x]][y]}{(f)}$. Let $\mathfrak{m}$ be a maximal ideal of $K[[x]][y]$ containing $(f)$. Then $\mathfrak{m}\cap K[[x]]=(x)$ and $\mathfrak m\supseteq(x)$. Hence $$\dfrac{K[[x]][y]}{\mathfrak{m}}\cong\frac{K[[x]][y]/(x)}{\mathfrak{m}/(x)}\cong\dfrac{K[y]}{\mathfrak{m}/(x)}$$ and $\frac{\mathfrak{m}}{(x)} \supseteq (\overline{f})$, where $\overline{f}$ denotes the image of $f$ in $\frac{K[[x]][y]}{(x)}$. But now it is easy to observe that $\overline{f}=\overline{y}^N$.
% $\overline{f}=f+(x)=y^N+(x)=\overline{y}^N$. 
From this we get $\frac{\mathfrak{m}}{(x)}= (\overline{y})$,
% $\textbf{m}/(x)=(\overline{y})$, 
hence $\mathfrak{m}=(x,y)$. By the isomorphism $\mathcal U\cong\frac{K[[x]][y]}{(f)}$, we conclude that the only maximal ideal of $\mathcal U$ is $(F,H)$.

Let us now prove that $\mathcal U$ and $\mathcal V$ have the same field of fractions, that is $Q(\mathcal U)=Q(\mathcal V)$. One inclusion is trivial as $\mathcal U\subseteq \mathcal V$. Given $g \in \mathcal{V}$ we observe that $F^ng \in \mathcal{U}$. Thus, 
 given $g/h\in Q(\mathcal V)$, we get $g/h=(F^Ng)/(F^Nh)\in Q(\mathcal U)$.
%given $\frac{g}{h}\in Q(\mathcal V)$, we get $\frac{g}{h}= \frac{F^Ng}{F^Nh}\in Q(\mathcal U)$.

We note then also that $\overline{\mathcal U}=\overline{\mathcal V}$. Indeed, we have the following chains of inclusions:

$$
K[[F]]\subseteq \mathcal U\subseteq \mathcal V\subseteq K[[t_1]]\times\cdots\times K[[t_d]]
$$
\noindent where the second and the third inclusions are integral as $\mathcal V$ is a finite $K[[F]]$-module and $K[[t_1]]\times\cdots\times K[[t_d]]\cong\overline{\mathcal V}$. Hence $\overline{\mathcal U}=\overline{\mathcal V}$ and $\overline{\mathcal U}$ is a finite $\mathcal U$-module.

Finally, since $F$ is an element of minimal value in $(F,H)$, we have $\mathcal B(\mathcal U)=\mathcal U\left[\frac{H}{F}\right]=\mathcal U[G]=\mathcal V$.
\end{proof}

%\begin{oss}
%Proposition \ref{blowup} together with the results in \cite[Section 3]{apery} gives us that $\{1,G\cdot F,\dots,G^{N-1}\cdot F^{N-1}\}$ is an \bc Ap\'{e}ry basis \ec for $\mathcal U$.
%\end{oss}

\begin{oss}\label{U}
Let $\mathcal{O}$ be the ring of a plane algebroid curve. Then %$\mathcal{O}$ is local and, 
 by Proposition \ref{discussion}, $\mathcal{O}=k[[x]]+k[[x]]y+k[[x]]y^2+\cdots+k[[x]]y^{N-1}$, where $v(x)=(e_1,\dots,e_d)=\min(v(\mathcal{O}\setminus\{\boldsymbol{0}\})$ 
%{(0,\dots,0)\})$
and $N=e_1+\cdots+e_d=e$ is the multiplicity of $\mathcal{O}$. Let $\mathcal B(\mathcal{O})$ be the blow-up ring of $\mathcal{O}$ and suppose $\mathcal B(\mathcal{O})$ to be semilocal. By Theorems \ref{semilocalring1part} and \ref{semilocalring2part}, choosing $\om= (e_1,\dots,e_d)$, we can write 
$$
\mathcal B(\mathcal{O})= K[[F]]+K[[F]]G + \cdots + K[[F]]G^{N-1}
$$ for opportune choices of $F$ and $G$.
Since $F$ can be any element of $\mathcal B(\mathcal{O})$ of value $\om$, %and since, for our purpose, $\om=(e_1,\dots,e_d)$, 
we can choose $F= x$ and get
$$
\mathcal B(\mathcal{O})= K[[x]]+K[[x]]G + \cdots + K[[x]]G^{N-1}.
$$
Finally, by Proposition \ref{blowup} we have that the local ring
$$
\mathcal U= K[[x]]+K[[x]]xG + \cdots + K[[x]]x^{N-1}G^{N-1}
$$
is the ring an algebroid curve and $\mathcal B(\mathcal U)=\mathcal B(\mathcal{O})$.
\end{oss}

\begin{prop}\label{U=O}
 The rings $\mathcal{O}$ and $\mathcal U$ considered in Remark \ref{U} are equal.
 \end{prop}
\begin{proof}
We need to prove that $y\in \mathcal U$ and $xG\in\mathcal{O}$. We know that $y/x \in\mathcal{O}[y/x]=\mathcal B(\mathcal{O})=\mathcal B(\mathcal U)$ and $v(y/x)$ is in the Ap\'ery set of $v(\mathcal B(\mathcal{O}))$ with respect to $\om$. Hence, by Theorem \ref{semilocalring2part}, $y/x =G^j+\phi(x,G)$ for some $\phi(x,G)$ of degree at most $j-1$ in $G$. We claim that $j=1$, that is $v(y/x)$ is in the first level of $\Ap(v(\mathcal B(\mathcal{O})), \om)$. 

Indeed, as recalled before Proposition \ref{blowup}, $\mathcal B(\mathcal{O})\subseteq K[[t_1]]\times\cdots\times K[[t_d]]$ and we can write $\mathcal B(\mathcal{O})=C_1\times C_2$ where $C_1$ and $C_2$ are the natural projection of $\mathcal B(\mathcal{O})$ over the sets of indexes $I_1=\left\{i\in\{1,\dots,d\}\ |\ v(y)_i=e_i\right\}$ and $I_2=\left\{i\in\{1,\dots,d\}\ |\ v(y)_i>e_i\right\}$, respectively. Both sets $I_1$ and $I_2$ are nonempty since we assumed $\mathcal B(\mathcal{O})$ to be not local.

Thus observe that $v(y/x)= (\boldsymbol{0}, \be) \in v(C_1) \times v(C_2)$. Observe that $C_2$ is local and generated as module by the powers of the image of $y/x$. %onto the ring corresponding to the set of indexes $I_2$.
By  Proposition \ref{3.8apery},  this implies that $\be$ is in the first level of the Ap\'ery set of $v(C_2)$. 
%\bc Since we can assume $v(y)$ to be a minimal nonzero element in $\Ap(v(\mathcal{O}),\om)$, then $\be$ is a minimal nonzero element in $v(C_2)$. 
Theorem \ref{thmnonlocallevels} yields $j= \lambda(v(y/x))= 1$ and therefore $y/x=G+\phi(x,G)$. 
 %Finally, 
 It follows that $y=xG+x\phi(x,G)$ and $\mathcal{O}=\mathcal{U}$ as $\phi(x,G)\in K[[F]] \subseteq \mathcal{O}\cap \mathcal{U}$.
\end{proof}

\begin{theorem}\label{final}
Let $\mathcal{O}$ be the ring of a plane algebroid curve and suppose its blow-up ring $\mathcal B(\mathcal{O})$ to be not local. Let $\om$ be the minimal nonzero element of $\mathcal{O}$.
Let $A_i$ and $A'_i$ denote the $i$-th levels of the Ap\'ery sets with respect to $\om$ of $v(\mathcal{O})$ and of $v(\mathcal B(\mathcal{O}))$, respectively. Then $A_i=A'_i+i\om$.
 \end{theorem}
\begin{proof}
%\bc check the notation \ec
We can describe $\mathcal{O}$ and $\mathcal{U}$ according to the notation used in Remark \ref{U}. Furthermore, denote by $\mathcal{O}_i$ the $K[[x]]$-submodule of $\mathcal{O}$ generated by $1, y, y^2, \ldots, y^i$ and, similarly, denote by $\mathcal{U}_i$ the $K[[x]]$-submodule of $\mathcal{U}$ generated by $1, xG, x^2G^2 \ldots, x^iG^i$. For the ring $\mathcal B(\mathcal{O})$ we adopt the notation of Theorems \ref{semilocalring1part} and \ref{semilocalring2part} setting $R= \mathcal B(\mathcal{O})$ and defining the subsets $R_i$ as for those theorems.

Thus, by \cite[Proposition 3.8]{apery}  and Proposition \ref{U=O} we have 
$$
A_i=\left\lbrace v(y^i+\phi_{i-1})\ |\  \phi_{i-1}\in\mathcal{O}_{i-1} \mbox{ and } v(y^i+\phi_{i-1})\notin v(\mathcal{O}_{i-1})\right\rbrace=
$$
$$
=\left\lbrace v(x^iG^{i} + \psi_{i-1}) \ | \ \psi_{i-1} \in \mathcal{U}_{i-1} \mbox{ and } v(x^iG^i + \psi_{i-1}) \notin  v(\mathcal{U}_{i-1}) \ec  \right\rbrace.
$$
By Theorem \ref{semilocalring2part}, we have 
$$
A'_i= \left\lbrace v(G^{i} + \varphi_{i-1}) \ | \ \varphi_{i-1} \in R_{i-1} \mbox{ and } v(G^i + \varphi_{i-1}) \notin  v(R_{i-1}) \ec  \right\rbrace.
$$
Hence in order to prove the theorem we can use exactly the same proof of \cite[Theorem 4.1]{apery}.
\end{proof}

\begin{ex}
\label{ex1}
Let us consider the ring 

$$\mathcal{O}=\frac{K[[X,Y]]}{(X^5-Y^2)\cap (X^7+X^5+3X^4Y-Y^3) \cap (X^5-X^2+2XY-Y^2)}$$
of a plane algebroid curve, which is parametrized by: $$\mathcal{O}=K[[(t^2,u^3,v^2),(t^5,u^5+u^7,v^2+v^5]]$$
If we compute the blow-up, we obtain:
$$\mathcal{O'}:=\mathcal{B(O)}=K[[(t^2,u^3,v^2),(t^5,u^2+u^4,1+v^3,]]=K[[(t^2,u^3),(t^3,u^2+u^4)]]\times K[[(v^2,v^3)]].$$
If we denote by $\mathcal{O'}_1:=K[[(t^2,u^3),(t^3,u^2+u^4)]]$ and $\mathcal{O'}_2:=K[[(v^2,v^3)]]$.
We have that the Apéry set of the semigroup $v(\mathcal{O'}_1))$ with respect to the element 2 is the set $\{0,3\}$ and $\lambda(0)=0,\lambda(2)=1,\lambda(3)=1,\lambda(4)=2$.
The Apéry set of $\mathcal{O'}_2$ with respect to the element (2,3) is depicted in figure \ref{fig:apery}.
\begin{figure}[H]
    \centering
    \includegraphics{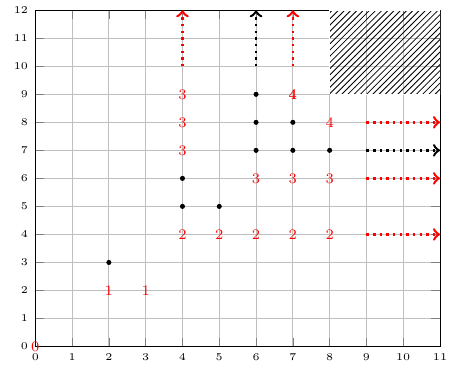}
    \caption{\footnotesize{The figure represents the Apéry Set of the semigroup $v(\mathcal{O'}_2)$}}
    \label{fig:apery}
\end{figure}
Using the method described in \cite[Theorem 4.5]{GMM2} we can determine the levels of the Apéry Set $A'$ of the ring $\mathcal{O'}$ with respect to the element $\om=(2,3,2)$. In this case $\ga=(5,5,1)$ and we have that:
\small{
\begin{eqnarray*}
A_0'&=&\{(0,0,0)\}\\    
A_1'&=&\{(0,0,2),(0,0,3),(2,2,0),(3,2,0)\}\\
A_2'&=&\{(0,0,\infty),(2,2,2),(2,2,3),(3,2,2),(3,2,3),(4,4,0),\\
&&(2,3,0),(5,4,0),(6,4,0),(7,4,0),(\infty,4,0)\}\\
A_3'&=&\{(2,2,\infty),(3,2,\infty),(2,3,3),(4,4,2),(4,4,3),(5,4,2),(5,4,3),(6,4,2),(6,4,3),\\
&&(7,4,2),(7,4,3),(\infty,4,2),(\infty,4,3),(4,5,0),(5,5,0),(4,6,0),(4,7,0),(4,8,0),\\
&&(4,\infty,0),(6,6,0),(7,6,0),(\infty,6,0)\}\\
A_4'&=&\{(4,4,\infty),(5,4,\infty),(6,4,\infty),(7,4,\infty),(\infty,4,\infty),(4,5,3),(5,5,3),(4,6,3),(4,7,2),\\
&&(4,7,3),(4,8,2),(4,8,3),(4,\infty,2),(4,\infty,3)(6,6,2),(6,6,3),(7,6,2),(7,6,3),(\infty,6,2),\\
&&(\infty,6,3),(\infty,8,0),(7,\infty,0)\}\\
A_5'&=&\{(4,7,\infty),(4,8,\infty),(4,\infty,\infty),(6,6,\infty),(7,6,\infty),(\infty,6,\infty),(6,7,3),(7,7,3),(\infty,7,3),\\
&&(6,8,3),(7,8,3),(6,\infty,3),(\infty,8,2),(\infty,8,3),(7,\infty,2),(7,\infty,3),(\infty,\infty,0)\}\\
A_6'&=&\{(\infty,8,\infty),(7,\infty,\infty),(\infty,\infty,3)\},
\end{eqnarray*}
}\normalsize
\noindent where, by convention, we say that an element of the form $\al=(\alpha_1,\alpha_2,\alpha_3)$ with $\alpha_i=\infty$ belongs to the set $A'_k$ if all the elements $\be$ with $\beta_i>\gamma_i+\omega_i$ and $\be_j=\alpha_j$, $j\neq i$ belong to the set $A'_k$.\\
Hence, using Theorem \ref{final}, we can compute the levels of the Ap\'ery set of the semigroup $v(\mathcal{O})$ with respect to the multiplicity (2,3,2) using the formula $A_i=A_i'+i(2,3,2)$, for $i\in\{0,\ldots,6\}$.

\end{ex}

%%%%%%%%%%%%%%%%%%%%%%%%%%%%%%%%%%%%%%%%%%%%%%%%%% Section
\section{Multiplicity trees of plane curve singularities}\label{sec:Multiplicities}
%%%%%%%%%%%%%%%%%%%%%%%%%%%%%%%%%%%%%%%%%%%%%%%%%%
 
Let $R \cong \OO_{1} \times \cdots \times \OO_c $ be a direct product of local
rings
$\OO_{j}$ ($1\le j \le c$) each one associated to a reduced plane algebroid curve
defined over an algebraically
closed field $K$.
Let us denote $\CC_1, \ldots, \CC_d$ (resp.
$\nu_1,\ldots, \nu_d$)
the branches of $R$ (resp. its valuations). 
For $j=1, \ldots, c$, let $S_j \subseteq \N^{d_j}$ denote the value semigroup
of $\OO_{j}$. For every $j$, $S_j$ is a local good semigroup (a numerical
semigroup if $d_j=1$).
The value semigroup of $R$ is $S= S_1 \times \cdots \times S_c \subseteq \N^{d}$
where $d = d_1 + \cdots + d_c$.
In this section will be useful to identify each semigroup $S_j\subset \N^{d_j}$ as a
the subsemigroup of $S$:
$S_j = \{0\}\times \cdots \times \{0\}\times S_j \times \{0\} \times \cdots  \times
\{0\}$.

\medskip

The {\it fine multiplicity \/} of $\OO_j$ is the minimal value
$\nu(x)\in S$ for $x\in \OO_j$ not unit. Notice that the identification of $S_j$
inside
$S$ implies that $\nu_i(x)\neq 0$ if and only if $\nu_i$ is a valuation of $\OO_j$.

The local rings $\OO_1, \ldots, \OO_c$ will be called the rings (or the points
following a more classical terminology) in the
$0$-neighbourhood of $R$.
Let $R^{(1)}\cong \OO_1^{(1)}\times \cdots \times \OO_c^{(1)}$ denote the ring in the
first neighbourhood of $R$, i.e. the ring produced after the blowing-up of $R$.
Notice that each ring $\OO_i^{(1)}$ is the product of a finite number of local
rings: the local rings (points) of the first-neighbourhood of $\OO_i$. All the local
rings of the
ring $R^{(1)}$ constitute the rings (or points) of the first-neighbourhood of $R$.

Recursively, for $j \ge 2$,
let $R^{(j)}\cong \OO_1^{(j)}\times \cdots \times \OO_c^{(j)}$ denote the ring in the
$j$-th neighbourhood of
$R$, i.e. the ring produced after $j$ blowing-ups of $R$ or equivalently the ring in
the first neighbourhood of $R^{(j-1)}$. As in the case $j=1$, the ring $R^{(j)}$ is
the product of finite number of local rings: the local rings (or the points) of the
$j$-neighbourhood of $R$. Notice that for $j$ big enough
$R^{(j)} \simeq \overline{R} \simeq K[[t_1]]\times \cdots \times K[[t_d]]$.

The whole set of local rings of the successive neighbourhoods is encoded as the set
of vertices ${\mathcal N}$ of a (infinite) graph ${\mathcal T}$ in a such a way that
two vertices corresponding to local rings $\OO$ and $\OO'$ are connected by an edge if
one of them is in the first neighbourhood of the other.
Thus, ${\mathcal T}$ is the disjoint union of $c$ graphs
$\mathcal T_1, \ldots, \mathcal T_c$, being $\mathcal T_i$ the graph
corresponding to the local ring $\OO_i$. Each ${\mathcal T_i}$ is a tree with root
in the vertex corresponding to $\OO_i$ and such that the $j$-th level of
${\mathcal T_i}$ consists of the vertices corresponding to the rings of the
$j$-neighbourhood of $\OO_i$.

The multiplicity graph of $R$ is the graph ${\mathcal T}$ with the additional
information of the fine multiplicity of each local ring attached as a weight of the
corresponding vertex. Although it is a tree only if $c=1$, we will refer to it as the
{\it multiplicity tree} of $R$ and we denote it by $\mathcal T (R)$ or simply
$\mathcal T$.

\medskip

The purpose of this section is the characterization of the admissible
multiplicity
trees of a plane curve singularity (not necessarily local) over an algebraically
closed field of arbitrary characteristic and to prove the
equivalence between the multiplicity tree, the semigroup of values and the suitable
sequences of multiplicities of each branch, together with the splitting numbers (equivalent
to the intersection multiplicities) between pair of branches.

The case $d\le 2$ and characteristic zero has been treated in \cite{apery}, however
the extension to any algebraically closed field, as well as
the sake of completeness make convenient to include it also here.
All the proofs of the results for $d=2$ (and characteristic zero) can be found in the
above reference.

As is well known, in positive characteristic the Newton-Puiseux theorem is not
valid. Therefore in this section we will systematically use the Hamburger-Noether
expansions which are valid in arbitrary characteristic. We have chosen to include
them in an almost self-contained way from Campillo's book \cite[Chapter II]{campillo}, where the
reader can find the precise proofs of the results we will use here. In some cases we
use some of the classical terminology of the treatment of singularities of complex
plane curves since from the point of view of the resolution and the combinatorial
invariants of the curves there is no substantial difference.

\begin{ex}
Let $$\mathcal{O}=\frac{K[[X,Y]]}{P_1\cap P_2 \cap P_3}$$
be a plane algebroid curve, parametrized by:
$$\mathcal{O}=K[[(t^2,u^3,v^2),(t^7,u^8+u^{10},v^4+v^7]],$$
with semigroup of values $S:=v(R)$ and multiplicity $\omega=(2,3,2)$.
We can compute the blow-up and multiplicity sequence:
$$\mathcal{O'}:=\mathcal{B(O)}=K[[(t^2,u^3,v^2),(t^5,u^5+u^7,v^2+v^5)]]$$
with semigroup of value $S'$ and multiplicity $\omega_1=(2,3,2)$.  $$\mathcal{O''}:=\mathcal{B(O')}=O''_1\times O''_2=K[[(t^2,u^3),(t^3,u^2+u^4)]]\times K[[(v^2,v^3)]]$$
with semigroups of values $S''_1\times S''_2:=v(O''_1)\times v(O''_2)$ and multiplicities $\omega_{2,1}=(2,2)$ and $\omega_{2,3}=2$; $$\mathcal{O'''}=\mathcal{B(O'')}=O'''_1\times O'''_2 \times O'''_3 =K[[t]]\times K[[u]]\times K[[v]]$$
with semigroups of values $S'''_1\times S'''_2\times S'''_3:=v(O'''_1)\times v(O'''_2)\times v(O'''_3)$.
In Figure \ref{fig:tree} are represented the blow-up tree $\mathcal{T}(R)$ and the multiplicity tree of semigroup $S$.

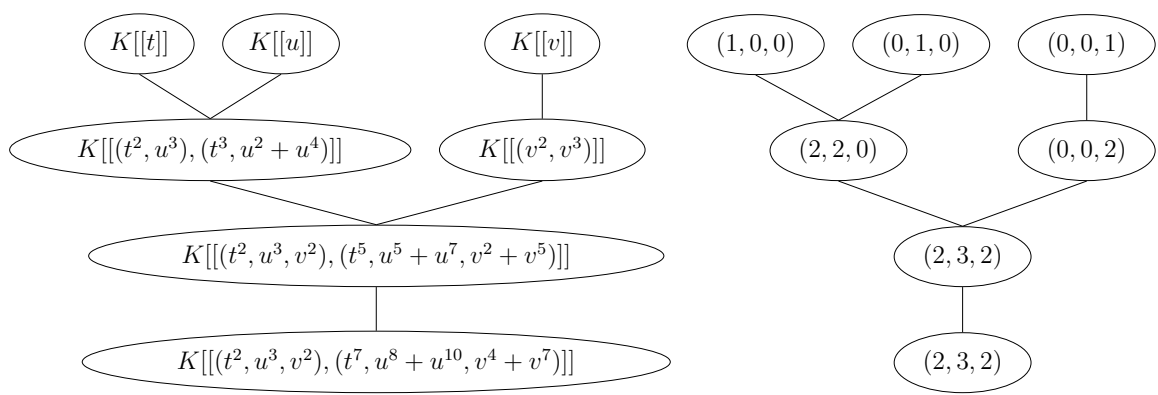
\begin{figure}[H]
\begin{subfigure}{.5\textwidth}
\centering
\begin{tikzpicture}[grow'=up,sibling distance=13pt,scale=.80]
		\tikzset{level distance=50pt,every tree node/.style={draw,ellipse}} \Tree [
.$K[[(t^2,u^3,v^2),(t^7,u^8+u^{10},v^4+v^7)]]$
        [ .$K[[(t^2,u^3,v^2),(t^5,u^5+u^7,v^2+v^5)]]$
[.$K[[(t^2,u^3),(t^3,u^2+u^4)]]$ [.$K[[t]]$ ] [.$K[[u]]$ ]] [.$K[[(v^2,v^3)]]$
[.$K[[v]]$ ]] ]]  \end{tikzpicture}
\end{subfigure}%
\begin{subfigure}{.5\textwidth}
\centering
    \begin{tikzpicture}[grow'=up,sibling distance=13pt,scale=.80]
		\tikzset{level distance=50pt,every tree node/.style={draw,ellipse}} \Tree [ .$(2,3,2)$
        [ .$(2,3,2)$ [.$(2,2,0)$ [.$(1,0,0)$ ] [.$(0,1,0)$ ]] [.$(0,0,2)$ [.$(0,0,1)$ ]] ]]  \end{tikzpicture}
\end{subfigure}%
\caption{On the left is represented the blow-up tree of R and on the right the multiplicity tree of the semigroup S.}
\label{fig:tree}
   \end{figure}
We want to show how to determine the semigroups of the tree, using the multiplicity tree of the semigroup S represented in Figure \ref{fig:tree}.
We have that $\Ap(S'''_1,2)=\Ap(S'''_2,2)=\{0,1\}$, hence we can determine the levels of Ap\'ery set  $\mathfrak{A}:=\Ap(S'''_1\times S'''_2,\omega_{2,1})$, which are
\begin{equation*}
\mathfrak{A}_0=\{(0,0)\}, \mathfrak{A}_1=\{(0,1),(1,0)\}, \mathfrak{A}_2=\{(0,\infty),(1,1),(\infty,0)\}, \mathfrak{A}_1=\{(\infty,1),(1,\infty)\}
\end{equation*}
 Using Theorem \ref{final} we have that $\Ap(S''_1,(2,2))_i=\mathfrak{A}_i+i(2,2)$ for all $i\in\{1,\ldots,4\}$.
Hence we can determine $S''_1=\Ap(S'',(2,2))+k(2,2)$ with $k\in \N$.
Starting by $\Ap(S'''_2,2)$, by Theorem \ref{final}, we obtain $\Ap(S''_2,2)_1=\{0\}$ and $\Ap(S''_2,2)_2=\{3\}$, determining the semigroup $S''_2$.
In Example \ref{ex1} we showed how to compute the levels of $\Ap(S',(2,3,2))$  knowing the levels of $\Ap(S'',(2,2))$ and $\Ap(S'',2)$; this Apéry set determines the semigroup $S'$. Using again Theorem \ref{final} we can determine the levels of $\Ap(S,(2,3,2))$ and the semigroup $S$.
\end{ex}

\subsection{Case $R$ irreducible (i.e. $c=1$ and $d_1=1$)}
Let $\CC (=\OO)$ be a plane irreducible algebroid curve (a branch) over an
algebraically closed field $K$ and
$\nu$ its valuation. The multiplicity tree is just a bamboo, so is equivalent to
the sequence of multiplicities  $\u e = (e_0,e_1,\ldots, e_n, \ldots )$
of $\CC$.
It is well known that the sequence of
multiplicities $\u e$ is an equivalent data to the semigroup $S = \nu(\CC)\subset \N$.
The sequence of multiplicities of a branch must be a (not strictly) decreasing
sequence satisfying also the following property:

\medskip
{\bf (Proximity)}
If $e_i > e_{i+1}$, let $e_{i}=q_i e_{i+1} + r_i, \ r_i<e_{i+1}$ be the Euclidean
division.
Then, $e_{i+j}=e_{i+1}$ for
$j=1,\ldots, q_i$, and, if $r_i\neq 0$ then $r_i:=e_{i+q_i+1}< e_{i+1}$.

\medskip

We will say that a sequence of positive integers
$\ee = (e_0, e_1, \ldots)$ is a {\it plane sequence} if is a decreasing one and
satisfies the Proximity relation above.
%It is well known that a sequence $\u e = (e_0,e_1,\ldots )$ is the multiplicity sequence
%of a plane branch if and only if it is plane.

Note that, as a consequence, for each $i\ge 0$ one has that
$e_i = \sum_{k=1}^{h(i)} e_{i+k}$ for a suitable $h(i)\ge 1$.
The
{\it restriction number}, $r(e_j)$, of an element $e_j$ of the sequence $\u e$ is
defined as the number of sums $e_i = \sum_{k=1}^{h(i)} e_{i+k}$ in which $e_j$
appears as a summand.
One has that $1\le r(e_j)\le 2$ and,
following the classical terminology of the infinitely near points,
if $r(e_j)=1$ we says that $\CC^{(j)}$ is a free point and if $r(e_j)=2$, $\CC^{(j)}$ is
a satellite point.

\subsubsection{Hamburger-Noether expansions}

Let $K$ be an algebraically closed field of arbitrary
characteristic, and let $\nu(g) = \text{ord}_t(g)$ be the valuation defined on the ring
of power series $K[[t]]$.

\begin{definition}
Let $x,y\in K[[t]]$ be such that $\nu(y)\ge \nu(x)\ge 1$.
The Hamburger-Noether (HN) expansion of $\{x,y\}$
is the finite set of expressions
\begin{equation}\label{HN}
z_{j-1} = \sum_{i=1}^{h_j} a_{ji}z_j^{i} + z_j^{h_j}z_{j+1}\; ; \; 0\le j \le r
\end{equation}
where $z_{-1}=y$, $z_0=x$, $a_{ji}\in K$, $h_{r}=\infty$ and
$z_1, \ldots, z_r\in K[[t]]$ are such that
$\nu(z_0)> \nu(z_1)> \cdots > \nu(z_r) \geq 1$.
\end{definition}

%We will say that a set of equalities as above (not necessarily defined starting
%on a concrete system $\{x,y\}$) is an Hamburger-Noether expansion.

\medskip

The HN expansion can be better understood from the recursive process of computation:
Being $\nu(y)\ge \nu(x)$, there exists a unique $a_{01}\in K$ such that
$\nu((y/x)-a_{01})>0$ (note that $a_{0 1}=0$ if and only if $\nu(y)>\nu(x)$). Let $y_1:=(y/x)-a_{01}$.
If $\nu(y_1)\ge \nu(x)$ we repeat the same operation with $\{x,y_1\}$.

In this way it is clear that we have one (and only one) of the next possibilities:

\begin{itemize}
\item[a)] After a finite number of steps, $h_0$,
we have $a_{0,1}, \ldots a_{0,h_0}\in K$ and $z_1\in K[[t]]$ such
that $\nu(z_1)<\nu(x)$ and
$
y = a_{01}x + a_{02}x^2 + \cdots + a_{0\,h_0}x^{h_0} + x^{h_0}z_1. \;
$
\item[b)]
We have an infinite series
$y=  a_{0 1}x + a_{0 2}x^2 + \cdots $ and the HN expansion is just this series.
\end{itemize}

Now, in case a) the process continue with the system $\{z_1,x\}$ in new row.
After a finite number, $r$, of steps a) we reach the case b) because
$\nu(z_i)< \nu(z_{i-1})$ for every $i$ and the valuation $\nu$ is discrete.

\medskip

\begin{oss}
It is useful to write the HN expansion in a more detailed way (called reduced form).
To do this, let $s_1<s_2 < \cdots <s_g=r$ be the ordered set of indices $j$ such that
$\nu(z_j) | \nu(z_{j-1})$ and for convenience we put also $s_0=0$. Then
in the row $j=s_i$
%$$
%z_{j-1} = \sum_{i=1}^{h_j} a_{j i}z_j^{i} + z_j^{h_j}z_{j+1}
%$$
there exists the minimum $k_i$ such that $a_{j\, k_i}\neq 0$ (note that
$k_i\ge 2$, because also $\nu(z_j)<\nu(z_{j-1})$).
In this way
the HN expansion (\ref{HN}) could be writen as:
\begin{equation}\label{HN2}
\begin{aligned}
(z_{-1})=y & = a_{0\,1}x + \cdots + a_{0\,h_0}x^{h_0} + x^{h_0}z_1 \\
(z_0) = x &= z_1^{h_1}z_2 \\
 & \cdots \\
z_{s_1-1} &= a_{s_1\,k_1}z_{s_1}^{k_1} + \cdots + a_{s_1\,h_{s_1}}z_{s_1}^{h_{s_1}} +
z_{s_1}^{h_{s_1}} z_{s_1+1} \\
z_{s_1} &=  z_{s_1+1}^{h_{s_1+1}} z_{s_1+2} \\
& \cdots    \\
z_{s_g-1} &= a_{s_g\,k_g}z_{s_g}^{k_g} + \cdots +
\end{aligned}
\end{equation}
where, for $i=1,\ldots, g$ one has  $a_{s_i\,k_i}\neq 0$.
\end{oss}

\subsubsection{Plane curves and HN expansions}

Let $ \CC= K[[x,y]]= K[[X,Y]]/P$ be a plane algebroid branch over $K$
%let $\OO$ be its local ring, $\OO = K[[x,y]]= K[[X,Y]]/P$,
and let $\text{\bf m} = (x,y)$ its maximal ideal.
Let $\k{\CC} \simeq K[[t]]$ be the integral closure of $\CC$ in its field of
fractions,
%so $\OO \hookrightarrow K[[t]]$ gives an irreducible (o primitive)
%parametrization of $\OO$.
so the valuation $\nu$ of $\CC$ is given by
$\nu(g) = \text{ord}_t (g(x(t), y(t)))$.
We assume that $\nu(x)\le \nu(y)$,
i.e. $x$ is a transversal parameter.

The {\it Hamburger-Noether expansion\/} of $\CC$ (with respect to $\{x,y\}$) is the
Hamburger-Noether expansion of $\{x,y\}\in K[[t]]$.
Notice that in this case it must be $\nu(z_r)=1$.

\medskip

Let  $\ee=(e_0, e_1\ldots )$ be the multiplicity sequence of $\CC$, one has
$e_0=\nu(x)$.
The blow-up of $\CC$ is the ring
$\CC^{(1)}= \CC [y/x]\subset \k{\CC}$, ${\mathfrak m}_1 = (x,y_1)$ is its maximal
ideal
and $\CC^{(1)} \simeq K[[x,y_1]]$. The coefficient $a_{0\,1}$ is the coordinate on the
exceptional divisor of the strict transform, i.e. $y-a_{0\,1}x$ is just the tangent
to $\CC$.
The multiplicity of  $\CC^{(1)}$ is $e_1 = \min\{\nu(x), \nu(y_1)\}$ and so
$e_1 = \nu(x)=e_0$ if $\nu(y_1)\ge \nu(x)$ and $e_1 = \nu(y_1)$ if
$\nu(y_1)<\nu(x)$.
In this way it is clear that the process of formation of the HN expansion exactly
reproduces the process of resolution of the singularity. In fact one has that (see
\cite[Proposition 2.2.9]{campillo}) the HN expansion of $\CC^{(1)}$ with respect to
$\{x,y_1\}$ is:
\begin{enumerate}
\item If $h_0>1$:
\begin{equation*}
\begin{aligned}
y_1 &= a_{0 2}x + \cdots + a_{0 h_0}x^{h_0-1} + x^{h_0-1}z_1\\
z_{j-1} &= \sum_{i=1}^{h_j} a_{ji}z_j^{i} + z_j^{h_j}z_{j+1}\; ; \; 1\le j \le r
\end{aligned}
\end{equation*}
\item If $h_0=1$:
\begin{equation*}
z_{j-1} = \sum_{i=1}^{h_j} a_{ji}z_j^{i} + z_j^{h_j}z_{j+1}\; ; \; 1\le j \le r
\end{equation*}
\end{enumerate}

In particular, let $n_i= \nu(z_i)$ be the values of the elements $z_i\in K[[t]]$, $0\le i\le r$.
Then the multiplicity sequence $\ee$ of $\CC$ is
$$
\u{n}=(n_0,\ldots, n_0, n_1,\ldots, n_1, \ldots, n_i,\ldots, n_i, \ldots, n_r, \ldots )
$$
where $n_i$ appears $h_i$ times.

%\begin{oss} Let $C$ be a plane algebroid branch over $K$,
%let $\OO = K[[x,y]]$ be its local ring and let
%\begin{equation*}
%z_{j-1} = \sum_{i=1}^{h_j} a_{ji}z_j^{i} + z_j^{h_j}z_{j+1}\; ; \; 0\le j \le r
%\end{equation*}
%be its HN expansion with respect to $\{x,y\}$ ($z_{-1}=y$,$z_0=x$).
%Let $\OO^{(1)}\simeq K[[x,y_1]]$ be quadratic transform of $\OO$ where
%$y_1= (y-a_{01}x)/x$. Then, the HN expansion of $\OO^{(1)}$ with respect to
%$\{x,y_1\}$ is:
%\begin{enumerate}
%\item If $h_0>1$:
%\begin{equation*}
%\begin{aligned}
%y_1 &= a_{0 2}x + \cdots + a_{0 h_0}x^{h_0-1} + x^{h_0-1}z_1\\
%z_{j-1} &= \sum_{i=1}^{h_j} a_{ji}z_j^{i} + z_j^{h_j}z_{j+1}\; ; \; 1\le j \le r
%\end{aligned}
%\end{equation*}
%\item If $h_0=1$:
%\begin{equation*}
%z_{j-1} = \sum_{i=1}^{h_j} a_{ji}z_j^{i} + z_j^{h_j}z_{j+1}\; ; \; 1\le j \le r
%\end{equation*}
%\end{enumerate}
%\end{oss}

\subsubsection{Multiplicity sequence and HN expansions.}

A set of formal expressions
\begin{equation}
z_{j-1} = \sum_{i=1}^{h_j} a_{ji}z_j^{i} + z_j^{h_j}z_{j+1}\; ; \; 0\le j \le r
\end{equation}
where $h_0,\ldots, h_{r-1}$ are positive integers, $h_{r}=\infty$ and
$a_{ji}\in K$ are such that $a_{j 1}=0$ if $j>0$, will be called an Hamburger-Noether
type expansion.

\medskip

Let us fix $r\ge 0$, and, if $r\ge 1$ let $1\le g\le r$. Let $h_0,\ldots, h_{r-1}$ be positive integers,
$0<s_1<\cdots < s_g = r$ and for $i=1,\ldots,g$ let $k_i$ integers such that
$2\le k_i\le h_{s_i}$.
Let
$H = (H_0,\ldots , H_r)$ be the sequence defined by
$H_j=[k_i,h_{s_i}]$ if $j=s_i$ and $H_j=h_j$ otherwise. The sequence $H$ defines an HN
type
expansion such that its reduced form is like (\ref{HN2}) with arbitrary coefficients
$a_{s_i \, k}\in K$, $0\le i\le g$, $k_i\le k\le h_{s_i}$,
$a_{s_i, k_i}\neq 0$. We say that this is an HN expansion of type $H$.

\begin{lemma}
\label{bijcorr}
There is a bijective correspondence between plane sequences $\ee$ and
finite sequences $H$ as above.
\end{lemma}

\begin{proof}
Let $\ee$ be a plane sequence. Let us write
$\ee = (n_0,\ldots, n_0,n_1\ldots, n_1,\ldots , n_r, \ldots)$ in such a way that
$n_i>n_{i+1}$ and let $h_i$ be the number of repetitions of $n_i$ ($h_r=\infty$).
Let $s_1<s_2<\cdots<s_g$ be the indexes $j$, $1\le j \le r$, such that
$n_j | n_{j-1}$ and $k_i = n_{j-1}/n_j \ge 2$ for $j=s_i$.
The proximity relation for  $n_{j-1}$ implies that $k_i\le h_{s_i}$.
Thus, we have defined a sequence $H(\ee)$.

Let $H$ be a sequence defined as above.
Then $H$ allows to define an unique sequence of positive
integers $(n_0,\ldots, n_0,n_1,\ldots, )$ starting with $n_j=1$ for $j\ge r$. Then
if $j<r$ define $n_{j-1} = h_j n_j + n_{j+1}$ if $j\neq s_i$ for all $i$ and
$n_{s_i-1} = k_i n_{s_i}$ if $j= s_i$.
Obviously this sequence $E(H)=\u{n}$ satisfies the proximity conditions and so is a
plane sequence. It is trivial that $E(-)$ and $H(-)$ are applications
inverse to each other.
\end{proof}

\begin{prop}
An Hamburger-Noether type expansion defines an unique plane irreducible curve
$\CC=K[[x,y]]$ with $\k{\CC}\simeq K[[z_r]]$ and whose HN sequence is the prefixed
one.

Moreover, let  $\ee$ be a plane sequence and
let $H(\ee)$ be a sequence defined as above for $\ee$. Then, an HN expansion of type
$H(\ee)$
defines an unique plane irreducible curve over $K$ such that its multiplicity
sequence is $\ee$.
\end{prop}

\begin{proof}
Let $x=z_0$, $y=z_{-1}$, $t=z_r$.
Performing the successive (inverse) substitutions we have a parametrization $x=
x(t)$, $y=y(t)$ and so we have a morphism
$\varphi: K[[X,Y]]\to K[[t]]$ defined by $\varphi(X)=x(t)$, $\varphi(Y)=y(t)$.
The ring $\CC=K[[x,y]]=K[[X,Y]]/\ker(\varphi)$ is the ring of an irreducible
algebroid plane curve.
Moreover, if $K((x,y))$ is the field of fractions of $\CC$, it is easy to see
(recursively) that $z_i\in K((x,y))$ for all $i$, in particular
$t=z_r\in K((x,y))$ and so $K((x,y)) = K((t))$, $\k{\CC} = K[[t]]$.
Obviously the HN expansion of $\CC$ with respect to $\{x,y\}$ is the one we started
with.

The second assertion is a trivial consequence of the first one and of Lemma \ref{bijcorr}.
\end{proof}

\begin{oss}
The relation between a plane sequence $\ee$ and the sequence $H(\ee)$ implies that
the free points (multiplicities) are exactly:
\begin{enumerate}
\item The first $h_0$ points of multiplicity $n_0$ and the first one of multiplicity
$n_1$
\item For each $t=1,\ldots,g-1$ the last $h_{s_t}-k_t\ge 0$ points of multiplicity
$n_{s_t}$ and the first one of multiplicity $n_{s_t +1}$.
For $t=g$ all the points of multiplicity $n_{s_g}=1$ but the first $k_g$.
\end{enumerate}
As a consequence, the free points (except the first one) are in a one to one correspondence with the coefficients
$\{a_{j i}\}$ of the HN expansion which are not forced to be zero.
So, for any choice of
$$
\{a_{s_t,i}\in K\; | \; 0\le t\le  g; k_{t}\le i\le h_{s_t}; a_{s_t,k_t}\neq 0\}
$$
one has a curve with multiplicity sequence $\ee$.

Moreover, the Euclidean algorithm for $n_{s_t}$ and $n_{s_t+1}$ determines
all the multiplicities $n_{i}$ ($s_t+2 \le i \le s_{t+1}$), the integers
 $h_{i}$ ($s_t+1 \le i <  s_{t+1}$) and also $k_{s_{t+1}}$. That is, all the
satellite points after the free point $n_{s_t+1}$ up to the next free point.

The rows $\{s_i: i=0,\ldots,g\}$ are called {\it the free rows} and the rest {\it the satellite rows} because of the distribution of free and satellite points.
\end{oss}

\subsection{Case of two branches (i.e. $d=2$).}

Let us assume first that the ring $R=\OO$ is a local one (i.e. $c=1$) with two
branches
$\CC$ and $\CC'$ ($d=2$).
Let ${\mathfrak p}$, ${\mathfrak p}'$ be the minimal prime ideals of $\OO$, then the
branch $\CC$ is $\CC=R/{\mathfrak p}$ and the
branch $\CC'$ is $\CC'=R/{\mathfrak p}'$.
%Let $\nu_i$ be the
%valuation defined by $R_i$ and denote by $\nu(g)=(\nu_1(g), \nu_2(g))$ for $g\in R$ a non-zero divisor, thus
%$\nu(R)\subset \N^2$ is just the (local) good semigroup $S$ of $R$.
Let
$\u e= (e_0,e_1,\ldots)$ (resp. $\ee'= (e'_0,e'_1,\ldots)$) be the
sequence of multiplicities of the branch $\CC$
(resp. $\CC'$).

\medskip

The {\it splitting number\/} of $\OO$ is defined as the biggest positive integer $k$
such that $\OO^{(k)}$ is local. Thus, one has that $\OO^{(k)}$ is local and
$\OO^{(k+1)}\simeq \CC^{(k+1)}\times {\CC'}^{(k+1)}$.
The multiplicity tree of $\OO$ is the result of identifying the bamboos of both
branches $\CC$ and $\CC'$ up to level $k$, the weights on the trunk are
the fine multiplicities of $\OO^{(j)}$, for $j\le k$, i.e.
$\{(e_j, e'_j); j=0,\ldots k\}$. After the splitting level $k$, i.e. for $j\ge k+1$,
the weights are the
fine multiplicity of $\CC^{(j)}$: $(m(\CC^{(j)}), 0) = (e_j,0)$ and the one of
${\CC'}^{(j)}$: $(0, m({\CC'}^{(j)})) = (0, e'_j)$.

Notice that, if $R$ is not local, (i.e. $d=2$ and $c=2$) then the splitting number is defined as
$k=-1$.

\medskip

The intersection multiplicity of $\CC$ and $\CC'$ is
given by the
Noether formula 
$[\CC,\CC'] = \sum_{j=0}^k e_j e'_j$  (an easy consequence of  the equality $[\CC, \CC']= e_0 e'_0 + [\CC^{(1)}, \CC'^{(1)}]$, see
\cite[Remark 2.3.2. iv)]{campillo}).  Thus, if one fix both sequences of
multiplicities
$\u e$ and $\u e'$, then the splitting number $k$ is equivalent to the intersection
multiplicity. As a consequence one has that the semigroup of values $S$ is an
equivalent data to the multiplicity tree.

\medskip

The splitting number (for a fixed pair of plane sequences $\u e$ and $\u e'$) is
not an arbitrary one.

\begin{definition}\label{admissible}
We will say that an integer $k\ge -1$ is {\it admissible} if $k=-1$ or $k\ge 0$ and it
satisfies the following properties:
\begin{enumerate}
\item $e_{i-1}=e_i$ if and only if $e'_{i-1}=e'_i$ for $i=1,\ldots, k-1$.
\item $r(e_j)=r(e'_j)$ for all $j\le k$.
\item
If $e_{k-1}>e_k$ then $e'_{k-1}=e'_k$
\item
If $r(e_k)=r(e'_k)=r(e_{k+1})=r(e'_{k+1})=2$ and if $e_{k-1}=e_k$, then
$e'_{k-1}>e'_k$
\end{enumerate}
\end{definition}

Notice that $k=-1$ is always admissible for any pair of plane
sequences.

\medskip

\begin{prop}\label{new}
Let $k\ge 0$ be an integer with the properties 1 and 2 of the above definition
\ref{admissible}. Then $k$ is admissible if and only if
either $k$ is maximal with the conditions 1 and 2 or
$r(e_{k+1})=r(e'_{k+1})=1$.
\end{prop}

\begin{proof}
Let us assume that $k$ is admissible and also that the conditions 1 and 2 are also
true for $k+1$. In particular (see property 3), $e_{k-1}=e_k$ and $e'_{k-1}=e'_k$.
Moreover, $r(e_{k+1})=r(e'_{k+1})$ and (see property 4) if is equal to 2 one reaches
a contradiction. Thus we have proved that, if $k$ is not maximal then
$r(e_{k+1})=r(e'_{k+1})=1$.

Let us show the sufficient condition. Firstly, note that the condition
$e_{k-1}>e_k$ implies that $r(e_{k+1})=2$. So, if $e_{k-1}>e_k$ and also
$e'_{k-1}>e'_k$ then $r(e_{k+1})=r(e'_{k+1})=2$ and $k$ is forced to be maximal. But
obviously this is not the case because 1 and 2 are also true for $k+1$. This proves
property 3.

To prove property 4, the hypothesis
$r(e_{k+1})=r(e'_{k+1})=2$ implies that $k$ must be maximal with properties 1 and 2.
So, if $e_{k-1}=e_k$ then must be $e'_{k-1}>e'_k$ and the proof is finished.
\end{proof}

As a consequence the properties of the definition can be expressed
in a somewhat simpler form in the following way:

\begin{definition}\label{admissible2}
We will say that an integer $k\ge -1$ is {\it admissible} if $k=-1$ or $k\ge 0$ and it
satisfies the following properties:
\begin{enumerate}
\item $e_{i-1}=e_i$ if and only if $e'_{i-1}=e'_i$ for $i=1,\ldots, k-1$.
\item $r(e_j)=r(e'_j)$ for all $j\le k$.
\item
Either $k$ is maximal with the conditions 1 and 2 or $r(e_{k+1})=r(e'_{k+1})=1$.
\end{enumerate}
\end{definition}

\begin{oss}
Notice that, if $e_{i-1}>e_i$ then $r(e_{i+1})=2$. As a consequence if $k$ is
admissible then:
\begin{enumerate}
\item
If $k$ is not maximal with properties 1 and 2, then  $r(e_{k+1})=r(e'_{k+1})=1$, i.e
both are free points and then $(e_{k-1}, e'_{k-1})=(e_k,e'_k)$.
However it is possible to have
$(e_{k-1}, e'_{k-1})=(e_k,e'_k)$ and
$r(e_{k+1})\neq r(e'_{k+1})$.
\item
The situation $e_{k-1}>e_k$ and $e'_{k-1}>e'_k$ is not possible. In particular
$e_k$ and $e'_k$ can not be simultaneously terminal free points.
\end{enumerate}
\end{oss}

\subsubsection{Intersection multiplicities with HN expansions}

Let $\OO\simeq K[[x,y]]$ be the local ring of a plane curve with two branches,
$\CC$ and  $\CC'$, let
$\ee = (e_0,e_1,\ldots)$ and  $\ee' = (e'_0,e'_1,\ldots)$ be the
multiplicity sequences of $\CC$ and $\CC'$.
Assume that $x$ is a transversal parameter for $\CC$ and $\CC'$.
Let $z_0 = z'_0=x$, $z_{-1}=z'_{-1}=y$ and let
\begin{equation}
\begin{aligned}
z_{j-1} & = \sum_{i=1}^{h_j} a_{ji}z_j^{i} + z_j^{h_j}z_{j+1}\; ; \; 0\le j \le r
\\
z'_{j-1} & = \sum_{i=1}^{h'_j} a'_{ji}(z'_j)^{i} + (z'_j)^{h'_j}z'_{j+1}\; ; \; 0\le j
\le r'
\end{aligned}
\end{equation}
be the HN expansions of $\CC$ and $\CC'$ with respect to $x,y$.

\medskip

Let $s$ be the largest integer such that
$h_j=h'_j$ for all $j<s$ and $a_{j i}=a'_{j i}$ for
$j<s$ and $i\le h_j$. Let $t\le \min\{h_s+1,h'_s+1\}$ be the largest integer for
which $a_{s i} =a'_{s i}$ for all $i<t$.

Note that, if $t < \min\{h_s+1,h'_s+1\}$ then
$a_{s\, t}\neq a'_{s\, t}$, in particular
$s= s_q$ for some $0\le q\le \min\{g,g'\}$. Otherwise,
$t = \min\{h_s+1,h'_s+1\}$ and necessarily $h_s\neq h'_s$.

\begin{prop}
With the above notations, let $S=\sum_{0}^{s-1} h_j n_j n'_j$. Then one has:
\begin{enumerate}
\item
The splitting number $k$ between $\CC$ and $\CC'$ is an equivalent data to the pair
$(s,t)$, in fact
$$
k = h_0+ h_1 + \cdots + h_{s-1} + t-1 \; .
$$
\item
The intersection multiplicity  $[\CC,\CC']$ is:
\begin{itemize}
\item[a)] If $t < \min\{h_s+1, h'_s+1\}$, then $[\CC,\CC']=S + t n_s n'_s$.
\item[b)] If $t=h'_s+1 < h_s+1 $, then $[\CC,\CC']=S + h'_s n_s n'_s + n'_{s+1}n_s$.
\item[c)] If $t=h_s+1 < h'_s+1 $, then $[\CC,\CC']=S + h_s n_s n'_s + n_{s+1}n'_s$.
\end{itemize}
\end{enumerate}
\end{prop}

\begin{proof}
One has that $k=0$ if and only if $a_{0 1}\neq a'_{0 1}$. Hence, this situation is
equivalent to $(s,t)=(0,1)$ and the equality follows.
The case $k>0$ is equivalent to $a_{0 1}=a'_{0 1}$ and the proof
follows by induction using the expression of the HN expansion of the strict transform
of a branch in terms of the one of $\CC$.

The equality of the intersection multiplicity is a consequence of the expression for
the splitting number or can be proved also by induction using that
$[\CC,\CC'] = n_0 n'_0 + [\CC^{(1)}, {\CC'}^{(1)}]$ (see \cite[2.3.2 and 2.3.3]{campillo}).
\end{proof}

%Let $k$ be an admissible integer for the plane sequences $\ee$ and $\ee'$ and let
%$H(\ee)$ and $H(\ee')$ be the finite sequences defined above for $\ee$ and $\ee'$.
%Let $\rho$ be the greatest integer such that
%$H_i = H'_i$ for $i<\rho$.
%
%
%Let $s$ and $t$ the integers such that
%$h_i=h'_i$ for $i\le s-1$ and
%$h_0 + \cdots + h_{s-1} < k \le h_0+\cdots +h_{s-1}+\min \{h_s,h'_s\}$.
%Let $t := k - (h_0+\cdots+h_{s-1}) +1$.
%Notice that $s$ and $t$ are unique and $t\le \min \{h_s+1,h'_s+1\}$.
%

\begin{prop}\label{2branches}
Let $\ee$, $\ee'$ be two plane sequences and let $k\ge -1$ be an admissible number
for
them. Let $\CC$ be a branch with multiplicity sequence $\ee$. Then, there exists a
branch $\CC'$ with multiplicity sequence $\ee'$ and
such that $k$ is the splitting number of the curve with branches $\CC$ and $\CC'$.
In particular, $k$ is the splitting number of a pair of branches with multiplicity
sequences $\ee$ and $\ee'$ if and only if $k$ is admissible.
\end{prop}

\begin{proof}
The case $k=-1$ is trivial. Let
$\CC$ be a branch with multiplicity sequence $\ee$ and HN expansion
\begin{equation*}
z_{j-1} = \sum_{i=1}^{h_j} a_{ji}z_j^{i} + z_j^{h_j}z_{j+1}\; ; \; 0\le j \le r\;
\end{equation*}
and let $k\ge 0$ be an admissible number for $\ee$ and $\ee'$.
Let
\begin{equation*}
z'_{j-1} = \sum_{i=1}^{h'_j} A'_{ji}(z'_j)^{i} + (z'_j)^{h'_j}z'_{j+1}\; ; \; 0\le j
\le r'
\end{equation*}
be an HN type expansion for $H(\ee')$ in which we see the symbols $\{A'_{ij}\}$ as
parameters to be determined.
If $k=0$ it suffices to fix $A'_{01} = a'_{01}\in K$ such that
$a'_{01}\neq a_{01}$. If $k>0$ then we fix $A'_{01}=a_{01}$. Now let
$\w{\ee}= (e_1,\ldots)$ and $\w{\ee}'= (e'_1,\ldots)$ and let
$\CC^{(1)}$ be the strict transform of $\CC$ by one blowing-up.
The multiplicity sequence of $\CC^{(1)}$ is $\w{\ee}$, $\w{\ee}'$ is a plane sequence
and $k-1$ is an admissible number for $\w{\ee}$ and $\w{\ee}'$.
By induction hypothesis, there exists a branch $D$ with multiplicity sequence
$\w{\ee}'$ and splitting number with $\CC^{(1)}$ equal to $k-1$.
The HN expansion of $D$ completed with $A'_{01}=a_{01}$ provides a branch $\CC'$ with
multilicity sequence $\ee'$ and such that its splitting number with $\CC$ is $k$.
\end{proof}

\subsection{General case.}

Let $R \cong \OO_{1} \times \cdots \times \OO_c $ be a direct product of local
rings
$\OO_{j}$ ($1\le j \le c$), each one associated to a reduced plane algebroid curve
defined over an algebraically
closed field $K$.
Let us denote $\CC_1, \ldots, \CC_d$
%(resp. $R_1, \ldots, R_d$ and
%$\nu_1,\ldots, \nu_d$)
the branches of $R$.
Let $\mathcal T$ be the multiplicity tree of $R$.
Take the notations given at the beginning of the section.
For each branch $\CC_i$, $i=1,\ldots,d$ one has its corresponding branch
${\mathcal T}^i$ of
$\mathcal T$ (i.e. a maximal completely ordered subtree of ${\cal T}$)
and so the sequence $\ee^i= (e_0,e_1,\ldots)$ of multiplicities of $\CC_i$.
For $i,j\in \{1,\ldots,d\}$ let $k_{i,j}+1$ be the length of the trunk of the
subtree of ${\mathcal T}$ given by $\CC_i$ and $\CC_j$, so
$k_{i,j}$ is just the splitting number of $\CC_1\cup \CC_2$.
The fact that $\mathcal T$ is the disjoint union of $c$ trees
implies some restrictions on the set of integers $\{k_{i,j}\}$, namely:

\begin{equation}\label{compatibility}
\text{Given } i,j,t\in\{1,\ldots,d\}.
\text{ If one has that }
k_{j,t} > k_{j,i} \text{ then  } k_{i,t} = k_{i,j}.
\end{equation}

Note that the condition (\ref{compatibility}) above is enough to construct a graph
${\mathcal T}(\{\ee^i\}, \{k_{i,j}\})$
by joining the $d$ sequences of integers $\{\ee^i; i=1,\ldots,d\}$ with the help of
the splitting vertices indicated by $\{k_{i,j}\}$.
(Pay attention that the graph is a tree if and only if $k_{i,j}\ge 0$ for any $i,j$).
More precisely:

\begin{lemma}\label{tree}
Let $E = \{\ee^i = (e^i_0, e^i_1, \ldots); i=1,\ldots, d\}$ be a set of sequences of
positive integers and
$\{k_{i,j}\ge -1\}$, $i, j\in \{1,\ldots,d\}, i\neq j$ an indexed set of integers
with $k_{i,j}=k_{j,i}$ and satisfying property (\ref{compatibility}).
Then there exists a weighted graph ${\mathcal T}={\mathcal T}(\{\ee^i\},
\{k_{i,j}\})$
such that the set of maximal completely ordered subgraphs of ${\cal T}$,
$\{{\cal T}^1, \ldots ,{\cal T}^d\}$, coincides with $E$ and for $i,j\in \{1,\ldots,
d\}$ the length of the trunk of ${\cal T}^i \cup {\cal T}^j \subset
{\cal T}$ is $k_{i,j} +1$.
\end{lemma}

\begin{proof}
The proof is easy by induction on the number of branches $d$. Otherwise,  we
can define directly the graph in the following way.
For each integer $t\ge 0$, let $i\sim_{t}j$ if and only if
$k_{i,j}\ge t$. If $k_{i,j}, k_{j,s}\ge t$, then by (\ref{compatibility}) one has
$k_{i,s}\ge \min \{k_{i,j}, k_{j.s}\}\ge t$. Thus, the relation $\sim_{t}$ is a
equivalence relation. For each equivalence class
$J_t$ we can take a vertex with weight $m(J_t)=(m_1,\ldots, m_d) \in \N^d$ defined as
$m_i = e^i_t$ if $i\in J_t$ and $m_i=0$ otherwise.
Notice that, if $t\ge \ell$ and  $i\sim_{t}j$, then $i\sim_{\ell}j$. Hence,
the result is the disjoint union of $c$ trees, each one with root in one
of the equivalence classes of $\sim_{0}$; in particular is a tree if and only if
$k_{i,j}\ge 0$ for all $i,j\in \{1,\ldots,d\}$.
\end{proof}

Adding to the Lemma the conditions of plane sequences and the admissibility one has:

\begin{prop}\label{multtree}
Let $\{\ee^i = (e^i_0, e^i_1, \ldots); i=1,\ldots, d\}$ be a set of sequences of
non-negative integers and
$\{k_{i,j}\ge -1\}$, $i, j\in \{1,\ldots,r\}, i\neq j$ an indexed set of integers
satisfying property (\ref{compatibility}). Let
${\mathcal T}={\mathcal T}(\{\ee^i\}, \{k_{i,j}\})$ be the weighted graph constructed
in Lemma \ref{tree}. Then there exists a plane curve with multiplicity tree
${\mathcal T}$ if and only if:
\begin{enumerate}
\item For $i=1,\ldots,r$, $\ee^i$ is a plane sequence.
\item For $i,j\in \{1,\ldots,r\}$, $i\neq j$, $k_{i,j}=k_{j,i}$ is an admissible
splitting number between the sequences $\ee^i$ and $\ee^j$.
%\item
%For each three different indices $\{i,j,t\}\subset \{1,\ldots,r\}$, one has that
%$\{k_{i.j}, k_{i,t}, k_{j,t}\}$ has at most two different elements and, moreover, if
%$k_{i,t} = k_{i,j}$ then $k_{j,t}\ge k_{i,j}$.
\end{enumerate}
\end{prop}

\begin{proof}
We will proceed by induction on the number of branches $d$. Notice that the case
$d\le 2$ is already known. Moreover, if there exists $i, j$ such that
$k_{i,j}=-1$, then the result is trivial because we can separate the
set of branches $\{1,\ldots, d\}$
in two parts $I, J$ such that $\# I, \# J <d$ and $k_{i,j}=-1$ for $i\in I$ and
$j\in J$. So, we can assume that $k_{i,j}\ge 0$ for any pair $i,j$, i.e. the
searched ring $R$ must be a local one.

Let us fix a branch $i=1$ and let $I = \{2,\ldots, d\}$. Let us assume that
$k_{1,2} \ge k_{1,i}$ for all $i\ge 2$.
Let ${\cal T}'$ be the sub-graph of ${\mathcal T}$ defined by the sequences
$\{\ee^i :  2\le i\le d\}$ and integers $\{k_{i,j} | i,j\ge 2\}$.
By the induction hypothesis, there exists a reduced curve $\CC'$ consisiting of the
branches $\CC_2, \ldots ,\CC_d$ such that the multiplicity tree of $\CC'$ is ${\cal T}'$.
Without loss of generality we can assume that
$\CC_i$ is given by an HN parametrization
$\varphi_i: (X,Y)\mapsto (x(t_i),y(t_i))$ in such
a way that, if $\varphi': K[[X,Y]]\to K[[t_2]]\times \cdots  \times K[[t_d]]$,
$\varphi'(f) = (\varphi_2(f), \ldots, \varphi_d(f))$,  then
$R' = K[[X,Y]]/\ker (\varphi)$ is the local ring of $\CC'$.

Let $\varphi_1: (X,Y) \mapsto (x(t_1), y(t_1))$ be an HN
parametrization of a branch $\CC_1$ such that its multiplicity sequence
coincides with $\ee^1$ and the splitting number with $\CC_2$ is equal to
$k_{1,2}$
(there exists by Proposition \ref{2branches}).
Being $k_{1,2}\ge k_{1,i}$ for all $i\ge 3$ it is clear that
the splitting number of $\CC_1$ and $\CC_i$ is equal to $k_{1,i}$.
Now we consider the map
$\varphi: K[[X,Y]]\to K[[t_1]]\times K[[t_2]]\times \cdots \times K[[t_d]]$ given by
$\varphi = (\varphi_1, \ldots, \varphi_d)$ and le $R = K[[X,Y]]/\ker(\varphi)$.
Then, one has that the multiplicity tree of $R$ coincides with
${\cal T}$ and the proof is finished.
\end{proof}

As a consequence, one also has:

\begin{theorem}\label{constructS}
Let $R \cong \OO_{1} \times \cdots \times \OO_c $ be a direct product of local
rings
$\OO_{j}$ ($1\le j \le c$), each one associated to a reduced plane algebroid curve
defined over an algebraically closed field $K$.
Let $\CC_1, \ldots, \CC_d$ be the set of branches of $R$.
The following elements are equivalent:
\begin{enumerate}
\item
The semigroup of values $S$ of $R$
\item The semigroups $S_i$, $1\le i\le d$, of each branch and the set of intersection
multiplicities
$\{[\CC_i,\CC_j] | 1\le i < j\le d\}$ between pairs of branches.
\item
The multiplicity tree ${\mathcal T}(R)$ of $R$.
\item
The set $E = \{\ee^i = (e^i_0, e^i_1, \ldots); i=1,\ldots, d\}$ of the multiplicity
sequences of the branches $\{\CC_i | 1\le i \le d\}$ plus the splitting numbers
$\{k_{i,j}\}$ between pairs of branches $\CC_i$, $\CC_j$; $1\le i < j\le d$.
\end{enumerate}
\end{theorem}

\section*{Acknowledgements}
Marco D'Anna and Vincenzo Micale are supported by the project PIA.CE.RI 2020-2022 Università di Catania - Linea 2 - "Proprietà locali e globali di anelli e di varietà algebriche". 

Marco D'Anna is also supported by the PRIN 2020 "Squarefree Gr\"obner degenerations, special varieties and related topics".

F\'elix Delgado is partially supported by grant PID2022-138906NB-C21 funded by MICIU/AEI/ 10.13039/501100011033 and by ERDF/EU. 

Lorenzo Guerrieri is supported by the NAWA Foundation grant Powroty "Applications of Lie algebras to Commutative Algebra" - PPN/PPO/2018/1/00013/U/00001.

Marco D'Anna and Lorenzo Guerrieri acknowledge the support of Indam-GNSAGA.

%The authors wish to thank Marco D'Anna for the interesting and helpful discussions about the content of this article.

 %\bf togliere quello che non serve dalla bibliografia \rm

%\section*{Data availability statement}
% Data sharing not applicable to this article as no datasets were generated or analysed during the current study.

\end{document}